\documentclass{amsart}
\usepackage{amssymb, amsmath, amsthm, graphics, comment, xspace, enumerate}
\usepackage{color}
\newtheorem{thm}{Theorem}[section]
\newtheorem{cor}[thm]{Corollary}
\newtheorem{lem}[thm]{Lemma}
\newtheorem{prop}[thm]{Proposition}
\theoremstyle{definition}
\newtheorem{defn}[thm]{Definition}
\newtheorem{example}[thm]{Example}
\theoremstyle{remark}
\newtheorem{rem}[thm]{Remark}
\numberwithin{equation}{section}

\begin{document}
\title[Metrical almost periodicity and applications]{Metrical almost periodicity and applications}

\author{M. Kosti\' c}
\address{Faculty of Technical Sciences,
University of Novi Sad,
Trg D. Obradovi\' ca 6, 21125 Novi Sad, Serbia}
\email{marco.s@verat.net}

{\renewcommand{\thefootnote}{} \footnote{
\\ \text{  }  \ \    {\it Key words and phrases.} $({\mathrm R},{\mathcal B},{\mathcal P},L)$-multi-almost periodic type functions, $({\mathrm R}_{X},{\mathcal B},{\mathcal P},L)$-multi-almost periodic type functions, Bohr $({\mathcal B},I',\rho,{\mathcal P})$-almost periodic type functions, composition principles,
abstract Volterra integro-differential equations.
\\  \text{  }  
The author is partially supported by grant 451-03-68/2020/14/200156 of Ministry
of Science and Technological Development, Republic of Serbia.}}

\begin{abstract}
In this paper, we analyze 
various classes of
multi-dimensional 
almost periodic type functions in general metric. The main classes of functions under our consideration are
$({\mathrm R},{\mathcal B},{\mathcal P},L)$-multi-almost periodic functions, $({\mathrm R}_{X},{\mathcal B},{\mathcal P},L)$-multi-almost periodic functions, Bohr $({\mathcal B},I',\rho,{\mathcal P})$-almost periodic functions and $({\mathcal B},I',\rho,{\mathcal P})$-uniformly recurrent functions.
We clarify the main structural properties for the introduced classes of almost periodic type functions and
provide some applications of our results to
the abstract Volterra integro-differential equations.
\end{abstract}
\maketitle

\section{Introduction and preliminaries}

The notion of almost periodicity was introduced by the Danish mathematician H. Bohr around 1924-1926 and later generalized by many other authors (for further information regarding almost periodic functions, we refer the reader to the research monographs \cite{besik}, \cite{diagana}, \cite{fink}, \cite{gaston}, \cite{nova-mono}, \cite{nova-selected}, \cite{188}, \cite{pankov} and \cite{30}). As clearly marked in \cite{marko-manuel-ap},
the theory of almost periodic functions depending on several real variables has not received so much attention of the authors compared with the theory of almost periodic functions depending on one real variable. The recent results about multi-dimensional almost periodic functions and 
multi-dimensional almost automorphic functions can be found in
\cite{nova-selected};
for more details about almost periodic functions on (semi-)topological groups, we refer the reader to \cite{marko-manuel-ap}, \cite{nova-selected} and the lists of references quoted therein. 

The classical definition of an almost periodic function goes as follows.
Suppose that $(X,\| \cdot \|)$ is a complex Banach space, and $F : {\mathbb R}^{n} \rightarrow X$ is a continuous function, where $n\in {\mathbb N}$. Then it is said that $F(\cdot)$ is almost periodic if and only if for each $\epsilon>0$
there exists $l>0$ such that for each ${\bf t}_{0} \in {\mathbb R}^{n}$ there exists ${\bf \tau} \in B({\bf t}_{0},l)\equiv \{ {\bf t} \in {\mathbb R}^{n} : |{\bf t}-{\bf t}_{0}|\leq l\}$ such that
\begin{align*}
\bigl\|F({\bf t}+{\bf \tau})-F({\bf t})\bigr\| \leq \epsilon,\quad {\bf t}\in {\mathbb R}^{n};
\end{align*}
here, $|\cdot -\cdot|$ denotes the Euclidean distance in ${\mathbb R}^{n}.$
This is equivalent to saying that for any sequence $({\bf b}_k)$ in ${\mathbb R}^{n}$ there exists a subsequence $({\bf a}_{k})$ of $({\bf b}_k)$
such that the sequence of translations $(F(\cdot+{\bf a}_{k}))$ converges in $C_{b}({\mathbb R}^{n}: X),$ the Banach space of all bounded continuous functions on ${\mathbb R}^{n},$ equipped with the sup-norm. Any trigonometric polynomial in ${\mathbb R}^{n}$ is almost periodic, and we know that a continuous function $F(\cdot)$ is almost periodic if and only if there exists a sequence of trigonometric polynomials in ${\mathbb R}^{n}$ which converges uniformly to $F(\cdot).$
Further on,
any almost periodic function $F(\cdot)$ is bounded, uniformly continuous, the mean value
$$
M(F):=\lim_{T\rightarrow +\infty}\frac{1}{(2T)^{n}}\int_{{\bf s}+K_{T}}F({\bf t})\, d{\bf t}
$$
exists and it does not depend on $s\in {\mathbb R}^{n};$ here and hereafter,  
$K_{T}:=\{ {\bf t}=(t_{1},t_{2},\cdot \cdot \cdot,t_{n}) \in {\mathbb R}^{n} :  |t_{i}|\leq T\mbox{ for }1\leq i\leq n\}.$ The Bohr-Fourier coefficient $F_{\lambda}\in X$ is defined by \index{Bohr-Fourier coefficient}
$$
F_{\lambda}:=M\Bigl(e^{-i\langle \lambda, {\bf \cdot}\rangle }F(\cdot)\Bigr),\quad \lambda \in {\mathbb R}^{n},
$$
and the Bohr spectrum of $F(\cdot),$ defined by\index{Bohr spectrum}
$$
\sigma(F):=\bigl\{ \lambda \in {\mathbb R}^{n} : F_{\lambda}\neq 0 \bigr\},
$$
is at most a countable set. 

Suppose now that $F : {\mathbb R}^n \rightarrow X$ is continuous. Then it is said that \index{function!almost automorphic} $F(\cdot)$ is almost automorphic if and only if for every sequence $({\bf b}_{k})$ in $\mathbb{R}^n$ there exist a subsequence $({\bf a}_{k})$ of $({\bf b}_{k})$ and a map $G : {\mathbb R}^n \rightarrow X$ such that
\begin{align} \label{first-equ1}
\lim_{k\rightarrow \infty}F\bigl( {\bf t}+{\bf a}_{k}\bigr)=G({\bf t})\ \mbox{ and } \  \lim_{k\rightarrow \infty}G\bigl( {\bf t}-a_{k}\bigr)=F({\bf t}).
\end{align}
If the above limits converge uniformly on compact subsets of ${\mathbb R}^{n}$, then we say that
$F(\cdot)$ is compactly almost automorphic. It is well known that an almost automorphic function is compactly almost automorphic if and only if it is uniformly continuous as well as that any almost periodic function is compactly almost automorphic. Furthermore, any automorphic function is bounded and it is well known that
a bounded continuous function $F(\cdot)$ is almost automorphic if and only if $F(\cdot)$ is Levitan $N$-almost periodic in the following sense: For each $\epsilon>0$ and $N>0,$
there exists $l>0$ such that for each ${\bf t}_{0} \in {\mathbb R}^{n}$ there exists ${\bf \tau} \in B({\bf t}_{0},l)\equiv \{ {\bf t} \in {\mathbb R}^{n} : |{\bf t}-{\bf t}_{0}|\leq l\}$ such that
\begin{align*}
\bigl\|F({\bf t}+{\bf \tau})-F({\bf t})\bigr\| \leq \epsilon,\mbox{ if }{\bf t}\in {\mathbb R}^{n} \mbox{ and }|{\bf t}|\leq N;
\end{align*}
we also say that $\tau$ is a Levitan $(\epsilon,N)$-period of function $F(\cdot).$

In the one-dimensional setting, the classes of bounded almost periodic functions and semi-periodic functions with the Hausdorff metric have been introduced by S. Sto\' inski \cite{stoinski1}-\cite{stoinski2} and later reconsidered by many other authors including
A. S. Dzafarov, G. M. Gasanov \cite{dz} and
A. P. Petukhov \cite{petko}. In  \cite{stoinski}, S. S. Sto\' inski has introduced and analyzed the class of unbounded 
almost periodic functions with the Hausdorff metric (cf. also \cite{stoja2}); real-valued functions almost periodic in variation and $L_{\alpha}$-almost periodic functions (for $\alpha \in (0,1),$ we obtain the class of H\"older almost periodic functions of order $\alpha$, while for $\alpha=1$ we obtain the class of Lipschitz almost periodic functions) have been analyzed by the same author in \cite{stoja} and \cite{stoja1}, respectively. Let us recall that any almost periodic in variation function (any H\"older almost periodic function of order $\alpha\in (0,1);$ any Lipschitz almost periodic function) $f: {\mathbb R} \rightarrow {\mathbb R}$ is almost periodic, while the converse statement is not true:

\begin{example}\label{stojko} (\cite{fasc,stoja1})
\begin{itemize}
\item[(i)] The sum of functions $f_{1}(\cdot)$ and $f_{2}(\cdot)$, where
$f_{1}(t):=\sin(\sqrt{2}\pi t),$ $t\in {\mathbb R},$ $f_{2}(t):=(t-k)\sin(\pi/(t-k)),$ $t\in (k,k+1)$ and $f_{2}(t):=0$, $t=k$ ($k\in {\mathbb Z}$),
is almost periodic but not almost periodic in variation. 
\item[(ii)] The sum of functions 
$f_{1}(\cdot)$ and $f_{2}(\cdot)$, where
$f_{1}(t):=\arcsin (t-4k),$ $t\in [4k-1,4k+1),$ $f_{1}(t):=\arcsin(-t+4k+2),$ $t\in [4k+1,4k+3),$ 
$f_{2}(t):=\arcsin (\sqrt{2}t-4k),$ $t\in [(4k-1)/\sqrt{2},(4k+1)/\sqrt{2}),$ $f_{2}(t):=\arcsin(-\sqrt{2}t+4k+2),$ $t\in [(4k+1)/\sqrt{2},(4k+3)/\sqrt{2})$ ($k\in {\mathbb Z}$),
is almost periodic in variation but not Lipschitz almost periodic.
\end{itemize}
\end{example}

The above-mentioned research articles can be viewed as a certain predecessor of this work.
Here, we investigate multi-dimensional 
almost periodic type functions in general metric; any notion of almost periodicity considered in the former paragraph as well as the notion of (IC)-almost periodicity, introduced and analyzed by M. Adamczak in \cite{marek}, is a very special case of the notion Bohr $({\mathcal B},I',\rho,{\mathcal P})$-almost periodicity introduced in Definition \ref{nafaks123456789012345} below (see also the notion of $C^{(\infty)}_{a}$-almost periodicity introduced and analyzed in \cite{stoinski-prim}). In contrast with the above-mentioned research studies, we primarily deal here with the vector-valued functions and the metric induced by the norm of a weighted $L^{p}$-space of functions or the norm of a weighted $C_{0}$-space of functions. Following this way of looking at things, we generalize several known classes of multi-dimensional (Stepanov) almost periodic type functions and multi-dimensional (Stepanov) almost automorphic type functions (\cite{nova-selected}).
We will always assume henceforth that $\emptyset \neq I \subseteq {\mathbb R}^{n},$ $(Y, \|\cdot\|_Y)$ is a complex Banach space,
$P \subseteq Y^{I},$ the space of all functions from $I$ into $Y,$  the zero function belongs to $P$, and 
${\mathcal P}=(P,d)$ is a metric space; if $f\in P,$ then we designate $\| f\|_{P}:=d(f,0).$ We can also slightly generalize our notion by requiring that ${\mathcal P}=(P,d)$ is a pseudometric space; we will not follow this approach henceforth for simplicity.

The organization and main ideas of this paper, which reconsiders and continues our former research studies A. Ch\' avez et al \cite{marko-manuel-ap}-\cite{marko-manuel-aa} and M. Fe\v ckan et al \cite{rho}, can be briefly described as follows. In Subsection \ref{karambita}, we recall the basic definitions and results from the theory of Lebesgue spaces with variable exponents.
After introducing the basic notation and terminology used throughout the paper, we remind the readers of the necessary definitions of $({\mathrm R}_{X},{\mathcal B})$-multi-almost periodic type functions, Bohr ${\mathcal B}$-almost periodic type functions and $({\mathrm R}_{X},{\mathcal B})$-multi-almost automorphic type functions (see Subsection \ref{maremare123}).

Section \ref{maremare} investigates $({\mathrm R}_{X},{\mathcal B},{\mathcal P},L)$-multi-almost periodic type functions. In Definition \ref{eovakoapp-new} and Definition \ref{eovakoapp1-new}, we introduce the classes of (strongly) $({\mathrm R},{\mathcal B},{\mathcal P},L)$-multi-almost periodic functions and (strongly) $({\mathrm R}_{X},{\mathcal B},{\mathcal P},L)$-multi-almost periodic functions.
A concrete motivation for the introduction of such function spaces comes from many reasons, and we will only present here the following important example from \cite{nova-selected}:

\begin{example}\label{herceg}
As is well known, the Euler equations in ${\mathbb R}^{n}$, where $n\geq 2,$ describe the motion of perfect incompressible
fluids. The essence of problem is to find 
the unknown functions $u = u(x, t) = (u^{1}(x, t), . . . , u^{n}(x, t))$ and $p = p(x, t)$ denoting the
velocity field and the pressure of the fluid, respectively, 
such that
\begin{align}\label{Euler}
\begin{split}
& \frac{\partial u}{\partial t}+(u\cdot \nabla)u +\nabla p=0
\quad \mbox{ in }{\mathbb R}^{n} \times (0,T), \\
& \mbox{div}\, u=0 \mbox{ in } {\mathbb R}^{n} \times (0,T),\\
& u(x,0)=u_{0}(x)\mbox{ in } {\mathbb R}^{n},
\end{split}
\end{align}
where $u_{0} = u_{0}(x) = (u^{1}_{0}
(x), . . . , u^{n}_{0}
(x))$
denotes the given initial velocity field. There are many results concerning the well-posedness of \eqref{Euler} in the case that the initial velocity field $u_{0}(x)$ belongs to some direct product of (fractional) Sobolev spaces. For our observation, it is crucial to remind the readers of the research article \cite{pak-park} by  H. C. Pak and Y. J. Park, who investigated the well-posedness of \eqref{Euler} in the case that the initial velocity field $u_{0}(x)$ belongs to the space $B_{\infty,1}^{1}({\mathbb R}^{n})^{n},$ where $B_{\infty,1}^{1}({\mathbb R}^{n})$ denotes the usual Besov space (see e.g., \cite[Definition 2.1]{sawada}); let the space $B_{\infty,1}^{0}({\mathbb R}^{n})^{n}$ be defined in the same way. It is well known that O. Sawada and R. Takada have proved, in \cite[Theorem 1.5]{sawada}, that the almost periodicity of function $u_{0}(x)$ in
${\mathbb R}^{n}$ implies that the solution $u(\cdot,t)$ of (\ref{Euler}) is
almost periodic in ${\mathbb R}^{n}$ for all $t\in [0,T].$ In \cite[Example 8.1.4]{nova-selected}, we have analyzed the situation in which
${\mathrm R}$ is an arbitrary collection of sequences in ${\mathbb R}^{n}$, and $u_{0}(\cdot)$ has the property that for each sequence $({\bf b}_{k})
$ in ${\mathrm R}$ there exists a subsequence $({\bf b}_{k_{l}})$ of $({\bf b}_{k})
$ such that the sequence of translations $(u_{0}(\cdot+{\bf b}_{k_{l}}))$ is convergent in the space $B_{\infty,1}^{0}({\mathbb R}^{n})^{n}$. Then 
for each sequence $({\bf b}_{k})
$ in ${\mathrm R}$ there exists a subsequence $({\bf b}_{k_{l}})$ of $({\bf b}_{k})
$ such that, for every $t\in [0,T],$ the sequence of translations $(u(\cdot+{\bf b}_{k_{l}},t))$ is convergent in the space $B_{\infty,1}^{0}({\mathbb R}^{n})^{n}.$ Albeit we will not analyze the properties of solutions in more detail here, we would like to note that the condition imposed on the initial value $u_{0}(x)$ is equivalent with the condition that the function $u_{0}(x)$ is $({\mathrm R},{\mathcal P})$-multi-almost periodic, where $P:=B_{\infty,1}^{0}({\mathbb R}^{n})^{n}$ and the metric $d$ is induced by the norm in $P.$
\end{example}

The convolution invariance of $({\mathrm R}_{X},{\mathcal B},{\mathcal P},L)$-multi-almost periodicity and the invariance of $({\mathrm R},{\mathcal B},{\mathcal P},L)$-multi-almost periodicity under the actions of infinite convolution products are investigated in Proposition \ref{convdiaggas}
and Theorem \ref{nova}, respectively.  The convergence of sequences of $({\mathrm R}_{X},{\mathcal B},{\mathcal P},L)$-multi-almost periodic functions in the metric space ${\mathcal P}$ 
is investigated in Proposition \ref{mackat}, while a composition principle in this direction is deduced in Theorem \ref{eovakoonakoap}. In Subsection \ref{autom}, we aim to generalize the spaces of multi-dimensional (Stepanov) almost automorphic functions using our approach of metrical almost periodicity; we also present here a completely new characterization of compactly almost automorphic functions.

Section \ref{maremare1} investigates Bohr $({\mathcal B},I',\rho, {\mathcal P})$-multi-almost periodic type functions. In Definition \ref{nafaks123456789012345}, we introduce the classes of 
Bohr $({\mathcal B},I',\rho,{\mathcal P})$-almost periodic functions and $({\mathcal B},I',\rho,{\mathcal P})$-uniformly recurrent functions.
Before proceeding any further, we would like to recall the following example from the introductory part of article \cite{marko-manuel-ap}, which presents a strong motivational factor for the introduction of such classes of functions:

\begin{example}\label{bosna}
In the homogenization theory, the crucial problem is the asymptotic behaviour of the solutions of the problem
\begin{align}\label{arhangel-sk}
\inf\Biggl\{ \int_{\Omega}f(hx,Du)+\int_{\Omega}\psi x : u(\cdot)\mbox{ Lipschitz continuous and }u=0\mbox{ on }\partial \Omega \Biggr\},
\end{align}
where $\emptyset \neq \Omega \subseteq {\mathbb R}^{n}$ is an open bounded set, $\psi(\cdot)$ is essentially bounded on $\Omega,$ and $f : {\mathbb R}^{n} \times {\mathbb R}^{n} \rightarrow [0,\infty)$ satisfies the usual Carath\' eodory conditions.
Under certain conditions, G. de Giorgi has proved, in \cite{de-giorgi}, that the values in \eqref{arhangel-sk} converges to
\begin{align}\label{ubicemote}
\inf\Biggl\{ \int_{\Omega}f_{\infty}(Du)+\int_{\Omega}\psi x : u(\cdot)\mbox{ Lipschitz continuous and }u=0\mbox{ on }\partial \Omega \Biggr\},
\end{align}
where $f_{\infty} : {\mathbb R}^{n} \rightarrow [0,\infty)$ is a convex function defined by
\begin{align*}
f_{\infty}(x)& :=\lim_{s\rightarrow \infty}s^{-n}\inf\Biggl\{ \int_{(0,s)^{n}}f(x,z+Du) : \\& u(\cdot)\mbox{ is Lipschitz continuous and }u=0\mbox{ on } \partial\Bigl((0,s)^{n}\Bigr)\Biggr\}.
\end{align*} 
In his research study
\cite{dearhangel0}, R. De Arcangelis has
assumed that $f(\cdot,z)\in L_{loc}^{1}({\mathbb R}^{n})$ for every $z\in {\mathbb R}^{n},$
$|z| \leq f(x,z)$ for a.e. $x\in {\mathbb R}^{n}$ and every $z\in {\mathbb R}^{n},$
and, for every $z\in {\mathbb R}^{n},$ the following holds: For every $\epsilon>0,$ there exists a finite real number $L_{\epsilon}>0$ such that, for every $x_{0}\in {\mathbb R}^{n},$
there exists $\tau \in x_{0}+B(0,L_{\epsilon})$ such that
$$
|f(x+\tau,z)-f(x,z)| \leq \epsilon(1+f(x,z)),\quad \mbox{ for a.e. }x\in  {\mathbb R}^{n}\mbox{ and every }z\in {\mathbb R}^{n}.
$$
Then we know that, for every open convex set $\Omega$ and for every essentially bounded function $\psi(\cdot)$ on $\Omega,$ the values in \eqref{arhangel-sk}
converge to the value in \eqref{ubicemote}.
The above-mentioned almost type periodicity of function $f(x,z)$ is closely connected (but not completely equivalent) with the notion of Bohr $(I',\rho,{\mathcal P})$-almost periodicity of function $f(x,z)$, where $I':=\{(x_{1},x_{2},...,x_{n},0,0,...,0) \in {\mathbb R}^{2n} : x_{i}\in {\mathbb R} \ \ (1\leq i\leq n)\},$ $\rho$ is the identity operator on ${\mathbb C},$ $P:=L^{\infty}_{\nu}({\mathbb R}^{2n} : {\mathbb C})$ with $\nu(x,z):=[1+f(x,z)]^{-1}$, $x,\ z\in {\mathbb R}^{n}$ and $d(f,g):=\|f-g\|_{P}$ for all $f,\ g\in P;$ see the next subsection for the notion of space $L^{\infty}_{\nu}({\mathbb R}^{2n} : {\mathbb C})$.
\end{example}

We present some structural results about Bohr $({\mathcal B},I',\rho,{\mathcal P})$-almost periodic functions and $({\mathcal B},I',\rho,{\mathcal P})$-uniformly recurrent functions in Theorem \ref{sade} (here we presenty a completely new characterization of almost automorphic functions), Proposition \ref{prcko}, Proposition \ref{als} and Proposition \ref{aerosmith}; in addition to the above, many illustrative examples are given here. In Subsection \ref{periodic123}, we 
explain how our approach can be employed to provide certain generalizations of multi-dimensional (Stepanov) $\rho$-almost periodic functions. Some applications of our results to the abstract Volterra integro-differential equations are given in Section \ref{some1234554321}; the final section of paper is reserved for some conclusions and final remarks about the considered function spaces.

\subsection{Lebesgue spaces with variable exponents
$L^{p(x)}$}\label{karambita}

Let $\emptyset \neq \Omega \subseteq {\mathbb R}^{n}$ be a nonempty Lebesgue measurable subset, and let 
$M(\Omega  : X)$ denote the collection of all measurable functions $f: \Omega \rightarrow X;$ $M(\Omega):=M(\Omega : {\mathbb R}).$ Furthermore, ${\mathcal P}(\Omega)$ denotes the vector space of all Lebesgue measurable functions $p : \Omega \rightarrow [1,\infty].$
For any $p\in {\mathcal P}(\Omega)$ and $f\in M(\Omega : X),$ we define
$$
\varphi_{p(x)}(t):=\left\{
\begin{array}{l}
t^{p(x)},\quad t\geq 0,\ \ 1\leq p(x)<\infty,\\ \\
0,\quad 0\leq t\leq 1,\ \ p(x)=\infty,\\ \\
\infty,\quad t>1,\ \ p(x)=\infty 
\end{array}
\right.
$$
and
$$
\rho(f):=\int_{\Omega}\varphi_{p(x)}(\|f(x)\|)\, dx .
$$
We define the Lebesgue space 
$L^{p(x)}(\Omega : X)$ with variable exponent
by
$$
L^{p(x)}(\Omega : X):=\Bigl\{f\in M(\Omega : X): \lim_{\lambda \rightarrow 0+}\rho(\lambda f)=0\Bigr\}.
$$
Equivalently,
\begin{align*}
L^{p(x)}(\Omega : X)=\Bigl\{f\in M(\Omega : X):  \mbox{ there exists }\lambda>0\mbox{ such that }\rho(\lambda f)<\infty\Bigr\};
\end{align*}
see, e.g., \cite[p. 73]{variable}.
For every $u\in L^{p(x)}(\Omega : X),$ we introduce the Luxemburg norm of $u(\cdot)$ by
$$
\|u\|_{p(x)}:=\|u\|_{L^{p(x)}(\Omega :X)}:=\inf\Bigl\{ \lambda>0 : \rho(u/\lambda)    \leq 1\Bigr\}.
$$ 
Equipped with the above norm, the space $
L^{p(x)}(\Omega : X)$ becomes a Banach space (see e.g. \cite[Theorem 3.2.7]{variable} for the scalar-valued case), coinciding with the usual Lebesgue space $L^{p}(\Omega : X)$ in the case that $p(x)=p\geq 1$ is a constant function.
Further on, for any $p\in M(\Omega),$ we define
$$
p^{-}:=\text{essinf}_{x\in \Omega}p(x) \ \ \mbox{ and } \ \ p^{+}:=\text{esssup}_{x\in \Omega}p(x).
$$
Set
$$
D_{+}(\Omega ):=\bigl\{ p\in M(\Omega): 1 \leq p^{-}\leq p(x) \leq p^{+} <\infty \mbox{ for a.e. }x\in \Omega \bigr \}.
$$
For $p\in D_{+}([0,1]),$ the space $
L^{p(x)}(\Omega : X)$ behaves nicely, with almost all fundamental properties of the Lebesgue space with constant exponent $
L^{p}(\Omega : X)$ being retained; in this case, we know that 
$$
L^{p(x)}(\Omega : X)=\Bigl\{f\in M(\Omega : X)  \, ; \,  \mbox{ for all }\lambda>0\mbox{ we have }\rho(\lambda f)<\infty\Bigr\}.
$$
Set 
$$
E^{p(x)}(\Omega :X):=\Bigl\{ f\in L^{p(x)}(\Omega :X) : \mbox{ for all }\lambda>0\mbox{ we have }\rho(\lambda f)<\infty\Bigr\};
$$
$E^{p(x)}(\Omega )\equiv E^{p(x)}(\Omega : {\mathbb C}).$ It is well known that 
$E^{p(x)}(\Omega :X)=
L^{p(x)}(\Omega :X),$ provided that $p\in D_{+}(\Omega )$ (see e.g. \cite{fan-zhao}).

We will use the following lemma (cf. \cite{variable} for the scalar-valued case):

\begin{lem}\label{aux}
\begin{itemize}
\item[(i)] (The H\"older inequality) Let $p,\ q,\ r \in {\mathcal P}(\Omega)$ such that
$$
\frac{1}{q(x)}=\frac{1}{p(x)}+\frac{1}{r(x)},\quad x\in \Omega .
$$
Then, for every $u\in L^{p(x)}(\Omega : X)$ and $v\in L^{r(x)}(\Omega),$ we have $uv\in L^{q(x)}(\Omega : X)$
and
\begin{align*}
\|uv\|_{q(x)}\leq 2 \|u\|_{p(x)}\|v\|_{r(x)}.
\end{align*}
\item[(ii)] Let $\Omega $ be of a finite Lebesgue's measure and let $p,\ q \in {\mathcal P}(\Omega)$ such $q\leq p$ a.e. on $\Omega.$ Then
 $L^{p(x)}(\Omega : X)$ is continuously embedded in $L^{q(x)}(\Omega : X),$ and the constant of embedding is less than or equal to 
$2(1+m(\Omega)).$
\item[(iii)] Let $f\in L^{p(x)}(\Omega : X),$ $g\in M(\Omega : X)$ and $0\leq \|g\| \leq \|f\|$ a.e. on $\Omega .$ Then $g\in L^{p(x)}(\Omega : X)$ and $\|g\|_{p(x)}\leq \|f\|_{p(x)}.$
\item[(iv)] (The dominated convergence theorem) Let $p\in {\mathcal P}(\Omega),$ and let $f_{k},\ f\in M(\Omega  : X)$ for all $k\in {\mathbb N}.$ If $\lim_{k\rightarrow \infty} f_{k}(x)=f(x)$ for a.e. $x\in \Omega $
and there exists a real-valued function $g\in E^{p(x)}(\Omega)$ such that $\|f_{k}(x)\|\leq g(x)$ for a.e. $x\in \Omega,$ then $\lim_{k\rightarrow \infty}\|f_{k}-f\|_{L^{p(x)}(\Omega :X)}=0.$
\end{itemize}
\end{lem}

For further information concerning the Lebesgue spaces with variable exponents
$L^{p(x)},$ we refer the reader to \cite{variable}, \cite{fan-zhao} and \cite{doktor}.

Suppose now that the set $I$ is Lebesgue measurable as well as that 
$\nu : I \rightarrow (0,\infty)$ is a Lebesgue measurable function. We deal with the following Banach space
$$
L^{p({\bf t})}_{\nu}(I: Y):=\bigl\{ u : I \rightarrow Y \ ; \ u(\cdot) \mbox{ is measurable and } ||u||_{p({\bf t})} <\infty \bigr\},
$$
where $p\in {\mathcal P}(I)$ and 
$$
\bigl\|u\bigr\|_{p({\bf t})}:=\bigl\| u({\bf t})\nu({\bf t}) \bigr\|_{L^{p({\bf t})}(I:Y)}.
$$
Suppose now that $\nu : I \rightarrow (0,\infty)$ is any function such that the function $1/\nu(\cdot)$ is locally bounded. We also deal with the Banach space $C_{0,\nu}(I : Y)$ consisting of all continuous functions $u : I \rightarrow
Y$ satisfying that $\lim_{|{\bf t}|\rightarrow \infty , {\bf t}\in I}
\|u({\bf t})\|_{Y}\nu({\bf t})=0$. Equipped with the norm
$\|\cdot\|:=\sup _{{\bf t}\in I}\|\cdot({\bf t})\nu({\bf t})\|_{Y},$ $C_{0,\nu}(I : Y)$ is a Banach space.
\vspace{1.6pt}

\noindent {\bf Notation and terminology.} Suppose that $X,\ Y,\ Z$ and $ T$ are given non-empty sets. Let us recall that a binary relation between $X$ into $Y$
is any subset
$\rho \subseteq X \times Y.$ 
If $\rho \subseteq X\times Y$ and $\sigma \subseteq Z\times T$ with $Y \cap Z \neq \emptyset,$ then
we define
$\sigma \cdot  \rho =\sigma \circ \rho \subseteq X\times T$ by
$$
\sigma \circ \rho :=\bigl\{(x,t) \in X\times T : \exists y\in Y \cap Z\mbox{ such that }(x,y)\in \rho\mbox{ and }
(y,t)\in \sigma \bigr\}.
$$
As is well known, the domain and range of $\rho$ are defined by $D(\rho):=\{x\in X :
\exists y\in Y\mbox{ such that }(x,y)\in X\times Y \}$ and $R(\rho):=\{y\in Y :
\exists x\in X\mbox{ such that }(x,y)\in X\times Y\},$ respectively; $\rho (x):=\{y\in Y : (x,y)\in \rho\}$ ($x\in X$), $ x\ \rho \ y \Leftrightarrow (x,y)\in \rho .$
If $\rho$ is a binary relation on $X$ and $n\in {\mathbb N},$ then we define $\rho^{n}
$ inductively. Set $\rho (X'):=\{y : y\in \rho(x)\mbox{ for some }x\in X'\}$ ($X'\subseteq X$).

We will always assume henceforth that $(X,\| \cdot \|)$, $(Y, \|\cdot\|_Y)$ and $(Z, \|\cdot\|_Z)$ are three complex Banach spaces, $n\in {\mathbb N},$ $\emptyset  \neq I \subseteq {\mathbb R}^{n},$
${\mathcal B}$ is a non-empty collection of non-empty subsets of $X,$ ${\mathrm R}$ is a non-empty collection of sequences in ${\mathbb R}^{n}$
and ${\mathrm R}_{\mathrm X}$ is a non-empty collection of sequences in ${\mathbb R}^{n} \times X$. We will always assume henceforth 
that
for each $x\in X$ there exists $B\in {\mathcal B}$ such that $x\in B.$ By
$L(X,Y)$ we denote the Banach space of all bounded linear operators from $X$ into
$Y;$ $L(X,X)\equiv L(X)$ and ${\mathrm I}$ denotes the identity operator on $Y.$ If ${\mathcal A}: D({\mathcal A}) \subseteq X \mapsto P(X)$ is a multivalued linear operator (see, e.g., \cite{nova-selected} for the notion), where $P(X)$ stands for the power set of $X,$
then its range and spectrum of ${\mathcal A}$ are denoted by
$R({\mathcal A})$ and $\sigma({\mathcal A}),$ respectively. By $B^{\circ}$ and $\partial B$ we denote the interior and the boundary of a subset $B$ of a topological space $X$, respectively. The Lebesgue measure in ${\mathbb R}^{n}$ is denoted by $m(\cdot),$ and
the Wright function of order $\gamma \in (0,1)$ is denoted by
$\Phi_{\gamma}(\cdot);$ see, e.g., \cite{nova-mono} for the notion. 
Define 
${\mathbb N}_{n}:=\{1,\cdot \cdot \cdot, n\}.$ If ${\mathrm A}$ and ${\mathrm B}$ are non-empty sets, then we define ${\mathrm B}^{{\mathrm A}}:=\{ f | f : {\mathrm A} \rightarrow {\mathrm B}\}.$
Unless stated otherwise, we will always assume henceforth that for each $B\in {\mathcal B}$ 
and ${\bf b} \in {\mathrm R}$ [$({\bf b;{\bf x}}) \in {\mathrm R}_{X}$] a collection $L(B; {\bf b})$ [$L(B; ({\bf b;{\bf x}}) )$] consists of certain subsets of $B,$ as well as that $L(B; {\bf b})$ [$L(B; ({\bf b;{\bf x}}) )$] contains all singletons $\{x\}$ when $x$ runs through $ B.$

\subsection{$({\mathrm R}_{X},{\mathcal B})$-Multi-almost periodic type functions, Bohr ${\mathcal B}$-almost periodic type functions and $({\mathrm R}_{X},{\mathcal B})$-multi-almost automorphic type functions}\label{maremare123}

The main aim of this subsection is to recall the basic definitions and results about $({\mathrm R}_{X},{\mathcal B})$-multi-almost periodic type functions, Bohr ${\mathcal B}$-almost periodic type functions and $({\mathrm R}_{X},{\mathcal B})$-multi-almost automorphic type functions.
In the following, we recall the notion of $({\mathrm R},{\mathcal B})$-multi-almost periodicity and the notion of $({\mathrm R}_{\mathrm X},{\mathcal B})$-multi-almost periodicity (see \cite{nova-selected} for further information in this direction):

\begin{defn}\label{eovakoap}\index{function!$({\mathrm R},{\mathcal B})$-multi-almost periodic}
Suppose that $\emptyset \neq I \subseteq {\mathbb R}^{n},$ $F : I \times X \rightarrow Y$ is a continuous function, and the following condition holds:
\begin{align}\label{lepolepo}
\mbox{If}\ \ {\bf t}\in I,\ {\bf b}\in {\mathrm R}\ \mbox{ and }\ l\in {\mathbb N},\ \mbox{ then we have }\ {\bf t}+{\bf b}(l)\in I.
\end{align}
Then 
we say that the function $F(\cdot;\cdot)$ is $({\mathrm R},{\mathcal B})$-multi-almost periodic if and only if for every $B\in {\mathcal B}$ and for every sequence $({\bf b}_{k}=(b_{k}^{1},b_{k}^{2},\cdot \cdot\cdot ,b_{k}^{n})) \in {\mathrm R}$ there exist a subsequence $({\bf b}_{k_{l}}=(b_{k_{l}}^{1},b_{k_{l}}^{2},\cdot \cdot\cdot , b_{k_{l}}^{n}))$ of $({\bf b}_{k})$ and a function
$F^{\ast} : I \times X \rightarrow Y$ such that
\begin{align*}
\lim_{l\rightarrow +\infty}F\bigl({\bf t} +(b_{k_{l}}^{1},\cdot \cdot\cdot, b_{k_{l}}^{n});x\bigr)=F^{\ast}({\bf t};x) 
\end{align*}
uniformly for all $x\in B$ and ${\bf t}\in I.$\index{space!$AP_{({\mathrm R},{\mathcal B})}(I \times X : Y)$}
\end{defn}

\begin{defn}\label{eovakoap1}\index{function!$({\mathrm R}_{\mathrm X},{\mathcal B})$-multi-almost periodic}
Suppose that $\emptyset  \neq I \subseteq {\mathbb R}^{n},$ $F : I \times X \rightarrow Y$ is a continuous function, and the following condition holds:
\begin{align}\label{lepolepo121}
\mbox{If}\ \ {\bf t}\in I,\ ({\bf b};{\bf x})\in {\mathrm R}_{\mathrm X}\ \mbox{ and }\ l\in {\mathbb N},\ \mbox{ then we have }\ {\bf t}+{\bf b}(l)\in I.
\end{align}
Then 
we say that the function $F(\cdot;\cdot)$ is $({\mathrm R}_{\mathrm X}, {\mathcal B})$-multi-almost periodic if and only if for every $B\in {\mathcal B}$ and for every sequence $(({\bf b;{\bf x}})_{k}=((b_{k}^{1},b_{k}^{2},\cdot \cdot\cdot ,b_{k}^{n});x_{k})_{k}) \in {\mathrm R}_{\mathrm X}$ there exist a subsequence $(({\bf b;{\bf x}})_{k_{l}}=((b_{k_{l}}^{1},b_{k_{l}}^{2},\cdot \cdot\cdot , b_{k_{l}}^{n});x_{k_{l}})_{k_{l}})$ of $(({\bf b};{\bf x})_{k})$ and a function
$F^{\ast} : I \times X \rightarrow Y$ such that
\begin{align*}
\lim_{l\rightarrow +\infty}F\bigl({\bf t} +(b_{k_{l}}^{1},\cdot \cdot\cdot, b_{k_{l}}^{n});x+x_{k_{l}}\bigr)=F^{\ast}({\bf t};x) 
\end{align*}
uniformly for all $x\in B$ and ${\bf t}\in I.$ \index{space!$AP_{({\mathrm R}_{\mathrm X},{\mathcal B})}(I \times X : Y)$} 
\end{defn}

The notion 
introduced in Definition \ref{eovakoap} is a special case of the notion introduced in Definition \ref{eovakoap1}. In order to see this, suppose that the function $F : I \times X \rightarrow Y$ is continuous.
Set
$$
{\mathrm R}_{\mathrm X}:=\bigl\{ b : {\mathbb N} \rightarrow {\mathbb R}^{n} \times X \, ;\, (\exists a\in {\mathrm R}) \, b(l)=(a(l);0)\mbox{ for all }l\in{\mathbb N}\bigr\}.
$$
Then it is clear that 
\eqref{lepolepo} holds for $I$ and ${\mathrm R}$ if and only if \eqref{lepolepo121} holds for $I$ and ${\mathrm R}_{X};$
furthermore,
$F(\cdot;\cdot)$ is $({\mathrm R},{\mathcal B})$-multi-almost periodic if and only if $({\mathrm R}_{\mathrm X},{\mathcal B})$-multi-almost periodic.

The notion of Bohr $({\mathcal B},I',\rho)$-almost periodic function
and the notion of $({\mathcal B},I',\rho)$-unformly recurrent function
have recently been introduced by 
M. Fe\v{c}kan et al. in \cite[Definition 2.1]{rho}:

\begin{defn}\label{nafaks123456789012345}
Suppose that $\emptyset  \neq I' \subseteq {\mathbb R}^{n},$ $\emptyset  \neq I \subseteq {\mathbb R}^{n},$ $F : I \times X \rightarrow Y$ is a continuous function, $\rho$ is a binary relation on $Y$ and $I +I' \subseteq I.$ Then we say that:
\begin{itemize}
\item[(i)]\index{function!Bohr $({\mathcal B},I',\rho)$-almost periodic}
$F(\cdot;\cdot)$ is Bohr $({\mathcal B},I',\rho)$-almost periodic if and only if for every $B\in {\mathcal B}$ and $\epsilon>0$
there exists $l>0$ such that for each ${\bf t}_{0} \in I'$ there exists ${\bf \tau} \in B({\bf t}_{0},l) \cap I'$ such that, for every ${\bf t}\in I$ and $x\in B,$ there exists an element $y_{{\bf t};x}\in \rho (F({\bf t};x))$ such that
\begin{align*}
\bigl\|F({\bf t}+{\bf \tau};x)-y_{{\bf t};x}\bigr\|_{Y} \leq \epsilon .
\end{align*}
\item[(ii)] \index{function!$({\mathcal B},I',\rho)$-uniformly recurrent}
$F(\cdot;\cdot)$ is $({\mathcal B},I',\rho)$-uniformly recurrent if and only if for every $B\in {\mathcal B}$ 
there exists a sequence $({\bf \tau}_{k})$ in $I'$ such that $\lim_{k\rightarrow +\infty} |{\bf \tau}_{k}|=+\infty$ and that, for every ${\bf t}\in I$ and $x\in B,$ there exists an element $y_{{\bf t};x}\in \rho (F({\bf t};x))$ such that
\begin{align*}
\lim_{k\rightarrow +\infty}\sup_{{\bf t}\in I;x\in B} \bigl\|F({\bf t}+{\bf \tau}_{k};x)-y_{{\bf t};x}\bigr\|_{Y} =0.
\end{align*}
\end{itemize}
\end{defn}

Concerning $({\mathrm R}_{X},{\mathcal B})$-multi-almost automorphic type functions, we need to recall the following definitions from \cite{marko-manuel-aa} (see also \cite{nova-selected}):

\begin{defn}\label{eovako}\index{function!(compactly) $({\mathrm R},{\mathcal B})$-multi-almost automorphic}
Suppose that $F : {\mathbb R}^{n} \times X \rightarrow Y$ is a continuous function. Then
we say that the function $F(\cdot;\cdot)$ is $({\mathrm R},{\mathcal B})$-multi-almost automorphic if and only if for every $B\in {\mathcal B}$ and for every sequence $({\bf b}_{k}=(b_{k}^{1},b_{k}^{2},\cdot \cdot\cdot ,b_{k}^{n})) \in {\mathrm R}$ there exist a subsequence $({\bf b}_{k_{l}}=(b_{k_{l}}^{1},b_{k_{l}}^{2},\cdot \cdot\cdot , b_{k_{l}}^{n}))$ of $({\bf b}_{k})$ and a function
$F^{\ast} : {\mathbb R}^{n} \times X \rightarrow Y$ such that
\begin{align}\label{love12345678}
\lim_{l\rightarrow +\infty}F\bigl({\bf t} +(b_{k_{l}}^{1},\cdot \cdot\cdot, b_{k_{l}}^{n});x\bigr)=F^{\ast}({\bf t};x)
\end{align}
and
\begin{align*}
\lim_{l\rightarrow +\infty}F^{\ast}\bigl({\bf t} -(b_{k_{l}}^{1},\cdot \cdot\cdot, b_{k_{l}}^{n});x\bigr)=F({\bf t};x),
\end{align*}
pointwisely for all $x\in B$ and ${\bf t}\in {\mathbb R}^{n}.$ If for each $x\in B$ the above limits converge uniformly on compact subsets of ${\mathbb R}^{n}$, then we say that
$F(\cdot ; \cdot)$ is compactly $({\mathrm R},{\mathcal B})$-multi-almost automorphic. 
\end{defn}

\begin{defn}\label{eovakoap1aa}\index{function!(compactly) $({\mathrm R}_{\mathrm X},{\mathcal B})$-multi-almost automorphic}
Suppose that $F : {\mathbb R}^{n} \times X \rightarrow Y$ is a continuous function.
Then
we say that the function $F(\cdot;\cdot)$ is $({\mathrm R}_{\mathrm X}, {\mathcal B})$-multi-almost automorphic if and only if for every $B\in {\mathcal B}$ and for every sequence $(({\bf b;{\bf x}})_{k}=((b_{k}^{1},b_{k}^{2},\cdot \cdot\cdot ,b_{k}^{n});x_{k})) \in {\mathrm R}_{\mathrm X}$ there exist a subsequence $(({\bf b;{\bf x}})_{k_{l}}=((b_{k_{l}}^{1},b_{k_{l}}^{2},\cdot \cdot\cdot , b_{k_{l}}^{n});x_{k_{l}}))$ of $(({\bf b};{\bf x})_{k})$ and a function
$F^{\ast} : {\mathbb R}^{n} \times X \rightarrow Y$ such that
\begin{align}\label{gospode}
\lim_{l\rightarrow +\infty}F\bigl({\bf t} +(b_{k_{l}}^{1},\cdot \cdot\cdot, b_{k_{l}}^{n});x+x_{k_{l}}\bigr)=F^{\ast}({\bf t};x)
\end{align}
and
\begin{align}\label{gospodine}
\lim_{l\rightarrow +\infty}F^{\ast}\bigl({\bf t} -(b_{k_{l}}^{1},\cdot \cdot\cdot, b_{k_{l}}^{n});x-x_{k_{l}}\bigr)=F({\bf t};x),
\end{align}
pointwisely for all $x\in B$ and ${\bf t}\in {\mathbb R}^{n}.$ 
\end{defn}

For our further work, it will be worth noting that the requirements in Definition \ref{eovako}, resp. Definition \ref{eovakoap1aa}, may cause only the limit equality \eqref{love12345678}, resp. \eqref{gospode}. We will call such functions half-$({\mathrm R},{\mathcal B})$-multi-almost automorphic, resp. half-$({\mathrm R}_{X},{\mathcal B})$-multi-almost automorphic.\index{function!half-$({\mathrm R},{\mathcal B})$-multi-almost automorphic}\index{function!half-$({\mathrm R}_{X},{\mathcal B})$-multi-almost automorphic}

We also need the following definitions from \cite{rho}:

\begin{defn}\label{drasko-presing}
Let ${\bf \omega}\in {\mathbb R}^{n} \setminus \{0\},$ $\rho$ be a binary relation on $Y$ 
and 
${\bf \omega}+I \subseteq I$. A continuous
function $F:I\rightarrow Y$ is said to be $({\bf \omega},\rho)$-periodic if and only if 
$
F({\bf t}+{\bf \omega})\in \rho(F({\bf t})),$ ${\bf t}\in I.
$ 
\end{defn}

\begin{defn}\label{drasko-presing1}
Let ${\bf \omega}_{j}\in {\mathbb R} \setminus \{0\},$ $\rho_{j}\in {\mathbb C} \setminus \{0\}$ is a binary relation on $Y$
and 
${\bf \omega}_{j}e_{j}+I \subseteq I$ ($1\leq j\leq n$). A continuous
function $F:I\rightarrow Y$ is said to be $({\bf \omega}_{j},\rho_{j})_{j\in {\mathbb N}_{n}}$-periodic if and only if 
$
F({\bf t}+{\bf \omega}_{j}e_{j})\in \rho_{j}(F({\bf t})),$ ${\bf t}\in I,
$ $j\in {\mathbb N}_{n}.$ 
\end{defn} \index{function!$({\bf \omega}_{j},\rho_{j})_{j\in {\mathbb N}_{n}}$-periodic}

If $\rho={\mathrm I}$ ($\rho_{j}={\mathrm I}$ for all $j\in {\mathbb N}_{n}$), then we also say that the function $F(\cdot)$ is ${\bf \omega}$-periodic, resp. $({\bf \omega}_{j})_{j\in {\mathbb N}_{n}}$-periodic.

\section{$({\mathrm R}_{X},{\mathcal B},{\mathcal P},L)$-Multi-almost periodic type functions}\label{maremare}

We start this section by introducing the following notion (in contrast with Definition \ref{eovakoap} and Definition \ref{eovakoap1}, we do not require the continuity of function $F(\cdot;\cdot)$ here a priori):

\begin{defn}\label{eovakoapp-new}\index{function!$({\mathrm R},{\mathcal B},{\mathcal P})$-multi-almost periodic}\index{function!strongly $({\mathrm R},{\mathcal B},{\mathcal P})$-multi-almost periodic}
Suppose that $\emptyset \neq I \subseteq {\mathbb R}^{n},$ $F : I \times X \rightarrow Y$ is a given function, and \eqref{lepolepo} holds.
Then 
we say that the function $F(\cdot;\cdot)$ is $({\mathrm R},{\mathcal B},{\mathcal P},L)$-multi-almost periodic, resp. strongly $({\mathrm R},{\mathcal B},{\mathcal P},L)$-multi-almost periodic in the case that $I={\mathbb R}^{n},$  if and only if for every $B\in {\mathcal B}$ and for every sequence $({\bf b}_{k}=(b_{k}^{1},b_{k}^{2},\cdot \cdot\cdot ,b_{k}^{n})) \in {\mathrm R}$ there exist a subsequence $({\bf b}_{k_{l}}=(b_{k_{l}}^{1},b_{k_{l}}^{2},\cdot \cdot\cdot , b_{k_{l}}^{n}))$ of $({\bf b}_{k})$ and a function
$F^{\ast} : I \times X \rightarrow Y$ such that, for every $l\in {\mathbb N}$ and $x\in B,$ we have 
$F(\cdot +(b_{k_{l}}^{1},\cdot \cdot\cdot, b_{k_{l}}^{n});x)-F^{\ast}(\cdot;x)\in P$ and, for every $B'\in L(B; {\bf b}),$
\begin{align}\label{love12345678app-new}
\lim_{l\rightarrow +\infty}\sup_{x\in B'} \Bigl\| F\bigl(\cdot +(b_{k_{l}}^{1},\cdot \cdot\cdot, b_{k_{l}}^{n});x\bigr)-F^{\ast}(\cdot;x)\Bigr\|_{P}=0,
\end{align} 
resp. $F(\cdot +(b_{k_{l}}^{1},\cdot \cdot\cdot, b_{k_{l}}^{n});x)-F^{\ast}(\cdot;x)\in P$, $F^{\ast}(\cdot -(b_{k_{l}}^{1},\cdot \cdot\cdot, b_{k_{l}}^{n});x)-F(\cdot;x)\in P,$ \eqref{love12345678app-new} 
and, for every $B'\in L(B; {\bf b}),$
\begin{align*}
\lim_{l\rightarrow +\infty}\sup_{x\in B'}\Bigl\| F^{\ast}\bigl(\cdot -(b_{k_{l}}^{1},\cdot \cdot\cdot, b_{k_{l}}^{n});x\bigr)-F(\cdot;x)\Bigr\|_{P}=0.
\end{align*}
\end{defn}

\begin{defn}\label{eovakoapp1-new}\index{function!$({\mathrm R}_{\mathrm X},{\mathcal B},{\mathcal P})$-multi-almost periodic}\index{function!strongly $({\mathrm R}_{\mathrm X},{\mathcal B},{\mathcal P})$-multi-almost periodic}
Suppose that $\emptyset  \neq I \subseteq {\mathbb R}^{n},$ $F : I \times X \rightarrow Y$ is a given function, and \eqref{lepolepo121} holds.  
Then 
we say that the function $F(\cdot;\cdot)$ is $({\mathrm R}_{\mathrm X}, {\mathcal B},{\mathcal P},L)$-multi-almost periodic,
resp. strongly $({\mathrm R}_{\mathrm X}, {\mathcal B},{\mathcal P},L)$-multi-almost periodic in the case that $I={\mathbb R}^{n},$
if and only if for every $B\in {\mathcal B}$ and for every sequence $(({\bf b;{\bf x}})_{k}=((b_{k}^{1},b_{k}^{2},\cdot \cdot\cdot ,b_{k}^{n});x_{k})_{k}) \in {\mathrm R}_{\mathrm X}$ there exist a subsequence $(({\bf b;{\bf x}})_{k_{l}}=((b_{k_{l}}^{1},b_{k_{l}}^{2},\cdot \cdot\cdot , b_{k_{l}}^{n});x_{k_{l}})_{k_{l}})$ of $(({\bf b};{\bf x})_{k})$ and a function
$F^{\ast} : I \times X \rightarrow Y$ such that, for every $l\in {\mathbb N}$ and $x\in B,$ we have $F(\cdot +(b_{k_{l}}^{1},\cdot \cdot\cdot, b_{k_{l}}^{n});x+x_{k_{l}})-F^{\ast}(\cdot;x)\in P$ and, for every $B'\in L(B; ({\bf b;{\bf x}}) )$,
\begin{align}\label{love12345678ap1}
\lim_{l\rightarrow +\infty}\sup_{x\in B'}\Bigl \| F\bigl(\cdot +(b_{k_{l}}^{1},\cdot \cdot\cdot, b_{k_{l}}^{n});x+x_{k_{l}}\bigr)-F^{\ast}(\cdot;x)\Bigr\|_{P}=0,
\end{align}
resp. $F(\cdot +(b_{k_{l}}^{1},\cdot \cdot\cdot, b_{k_{l}}^{n});x+x_{k_{l}})-F^{\ast}(\cdot;x)\in P$, $F^{\ast}(\cdot -(b_{k_{l}}^{1},\cdot \cdot\cdot, b_{k_{l}}^{n});x-x_{k_{l}})-F(\cdot;x)\in P,$ \eqref{love12345678ap1} holds, and, for every $B'\in L(B; ({\bf b;{\bf x}}) )$,
\begin{align}\label{love12345678ap1s}
\lim_{l\rightarrow +\infty}\sup_{x\in B'}\Bigl \|F^{\ast}\bigl(\cdot -(b_{k_{l}}^{1},\cdot \cdot\cdot, b_{k_{l}}^{n});x-x_{k_{l}}\bigr)-F(\cdot;x)\Bigr\|_{P}=0.
\end{align}
\end{defn}

As above, we may simply conclude that 
the notion 
introduced in Definition \ref{eovakoapp-new} is a special case of the notion introduced in Definition \ref{eovakoapp1-new}. We will omit the term ``${\mathcal B}$'' from the notation if $X=\{0\},$ i.e., if we consider functions of the form $F : I \rightarrow Y.$

\begin{rem}\label{net}
\begin{itemize}
\item[(i)] In place of sequences, we can consider general nets here (see e.g., \cite[p. 9]{meise}); we will not consider such a general notion in this paper.
\item[(ii)] The notion introduced in Definition \ref{eovakoap}, resp. Definition \ref{eovakoap1}, is a special case of Definition \ref{eovakoapp-new}, resp. Definition \ref{eovakoapp1-new}, with
$P=l_{\infty}(I:Y):=\{ f : I \rightarrow Y : \sup_{{\bf t}\in I}\|f({\bf t})\|_{Y}<+\infty\},$ $d(f,g):=\sup_{{\bf t}\in I}\|f({\bf t})-g({\bf t})\|_{Y}$
and
${\mathcal P}:=(l_{\infty}(I:Y),d);$ see also \cite[Example 3.4, p. 14]{meise}. It suffices to require that $B\in L(B; {\bf b})$ [$B\in L(B; ({\bf b;{\bf x}}) )$] for all $B\in {\mathcal B}$ and ${\bf b}\in {\mathrm R}$ [$({\bf b;{\bf x}}) \in {\mathrm R}_{X}$].
\item[(iii)] Concerning the supremum formula for $({\mathrm R},{\mathcal B})$-multi-almost periodic functions, which has been established in \cite[Proposition 2.6]{marko-manuel-ap}, we will only note that its proof is based on the following argument: Suppose that $F: I \times X \rightarrow Y$ is $({\mathrm R},{\mathcal B},{\mathcal P},L)$-multi-almost periodic and \eqref{love12345678app-new} holds. If $P$ has a linear vector structure, then for each positive real number $\epsilon>0$ we have the existence of a number $l_{0}\in {\mathbb N}$ such that
$$
\sup_{x\in B'}\Bigl\| F(\cdot +b_{k_{l_{0}}};x)-F(\cdot +b_{k_{l}};x) \Bigr\|_{P}\leq \epsilon,\quad l\geq l_{0}.
$$
If ${\bf t}-b_{k_{l_{0}}}+b_{k_{l}}\in I$ for all ${\bf t}\in I,$ and $\| f(\cdot)\|_{P}=\| f(\cdot+\tau)\|_{P}$ for all $f\in P$ and $\tau \in {\mathbb R}^{n}$ with $I+\tau\subseteq I,$ then the above implies
$$
\sup_{x\in B'}\Bigl\| F(\cdot;x)-F(\cdot -b_{k_{l_{0}}}+b_{k_{l}};x) \Bigr\|_{P}\leq \epsilon,\quad l\geq l_{0}.
$$
\item[(iv)] The usually considered case is that one in which ${\mathrm R},$ resp. ${\mathrm R}_{X},$ is equal to the collection of all sequences in ${\mathbb R}^{n},$ resp. ${\mathbb R}^{n}\times X.$ In \cite[Example 8.1.3]{nova-selected}, we have provided an illustrative example showing the importance of case in which
${\mathrm R}$ is not equal to the collection of all sequences in ${\mathbb R}^{n}.$
\end{itemize}
\end{rem}

Suppose that $k\in {\mathbb N}$, $F_{i} : I \times X \rightarrow Y_{i},$
$P_{i} \subseteq Y^{I},$ the zero function belongs to $P_{i}$, 
${\mathcal P}_{i}=(P_{i},d_{i})$ is a metric space, and $Y_{i}$ is a complex Banach space ($1\leq i\leq k$). We define
the function $(F_{1},\cdot \cdot \cdot, F_{k}) : I \times X \rightarrow Y_{1}\times \cdot \cdot \cdot \times Y_{k}$ by
$$
(F_{1},\cdot \cdot \cdot, F_{k})({\bf t} ;x):=(F_{1}({\bf t};x) , \cdot \cdot \cdot, F_{k}({\bf t};x) ),\quad {\bf t} \in I,\ x\in X,
$$
as well as the product $P:=P_{1}\times \cdot \cdot \cdot \times P_{k}$ of metric spaces $P_{1},..., P_{k};$ let us recall that the metric $d(\cdot, \cdot)$ on $P$ is given by 
$$
d\bigl( (f_{1},...,f_{k}) , (g_{1},...,g_{k}) \bigr):=\sum_{i=1}^{k}d_{i}\bigl( f_{i},g_{i} \bigr),\quad f_{i},\ g_{i}\in P_{i}\ \ (1\leq i\leq k).
$$
The proof of following extension of \cite[Proposition 2.4]{marko-manuel-ap} is simple and therefore omitted.

\begin{prop}\label{kursk-kursk}
\begin{itemize}
\item[(i)] Suppose that $k\in {\mathbb N}$, $\emptyset \neq I \subseteq {\mathbb R}^{n},$ \eqref{lepolepo} holds and for any sequence  which belongs to 
${\mathrm R}$ we have that any its subsequence also belongs to ${\mathrm R}.$ 
If the function $F_{i}(\cdot;\cdot)$ is $({\mathrm R},{\mathcal B},{\mathcal P}_{i},L)$-multi-almost periodic, resp. strongly $({\mathrm R},{\mathcal B},{\mathcal P}_{i},L)$-multi-almost periodic in the case that $I={\mathbb R}^{n},$ for $1\leq i\leq k$, then the function $(F_{1},\cdot \cdot \cdot, F_{k})(\cdot;\cdot)$ is 
$({\mathrm R},{\mathcal B},{\mathcal P},L)$-multi-almost periodic, resp. strongly $({\mathrm R},{\mathcal B},{\mathcal P},L)$-multi-almost periodic.
\item[(ii)] Suppose that $k\in {\mathbb N}$, $\emptyset \neq I \subseteq {\mathbb R}^{n},$ \eqref{lepolepo121} holds and for any sequence  which belongs to 
${\mathrm R}_{X}$ we have that any its subsequence also belongs to ${\mathrm R}_{X}.$ 
If the function $F_{i}(\cdot;\cdot)$ is $({\mathrm R}_{X},{\mathcal B},{\mathcal P}_{i},L)$-multi-almost periodic, resp. strongly $({\mathrm R}_{X},{\mathcal B},{\mathcal P}_{i},L)$-multi-almost periodic in the case that $I={\mathbb R}^{n},$ for $1\leq i\leq k$, then the function $(F_{1},\cdot \cdot \cdot, F_{k})(\cdot;\cdot)$ is 
$({\mathrm R}_{X},{\mathcal B},{\mathcal P},L)$-multi-almost periodic, resp. strongly $({\mathrm R}_{X},{\mathcal B},{\mathcal P},L)$-multi-almost periodic.
\end{itemize}
\end{prop}

The convolution invariance of $({\mathrm R}_{\mathrm X},{\mathcal B})$-multi-almost periodicity has recently been analyzed in \cite[Proposition 2.5]{marko-manuel-ap}. In general case,
the convolution invariance of $({\mathrm R}_{X},{\mathcal B},{\mathcal P},L)$-multi-almost periodicity, resp. strong $({\mathrm R}_{X},{\mathcal B},{\mathcal P},L)$-multi-almost periodicity,
can be analyzed only if we assume that the metric space ${\mathcal P}$ has some extra features (the convolution invariance of $({\mathrm R},{\mathcal B},{\mathcal P},L)$-multi-almost periodicity, resp. strong $({\mathrm R},{\mathcal B},{\mathcal P},L)$-multi-almost periodicity,
can be analyzed similarly). Concerning this issue, we will  only formulate the following result without proof (see also the proof of Theorem \ref{nova} below and \cite{nova-selected} for many similar results of this type):

\begin{prop}\label{convdiaggas}
Let $P:=C_{0,\nu}({\mathbb R}^{n} :Y)$ and $d(f,g):=\| f-g\|_{C_{0,\nu}({\mathbb R}^{n} :Y)}$ for all $f,\ g\in P,$ and let there exist a positive real number $c>0$ such that $\nu({\bf t})\geq c$ for all ${\bf t}\in {\mathbb R}^{n}.$
Suppose that $h\in L^{1}({\mathbb R}^{n}),$ the function $F(\cdot ; \cdot)$ is $({\mathrm R}_{X},{\mathcal B},{\mathcal P},L)$-multi-almost periodic, resp. strongly $({\mathrm R}_{X},{\mathcal B},{\mathcal P},L)$-multi-almost periodic, as well as, for every $B\in {\mathcal B},$ $B'\in L(B; ({\bf b};{\bf x})),$ $l\in {\mathbb N}$ and
for every sequence $(({\bf b;{\bf x}})_{k}=((b_{k}^{1},b_{k}^{2},\cdot \cdot\cdot ,b_{k}^{n});x_{k})_{k}) \in {\mathrm R}_{\mathrm X}$, we have that 
the mapping ${\bf t}\mapsto F({\bf t}+b_{k_{l}}; x+x_{k_{l}}),$ ${\bf t}\in {\mathbb R}^{n}$ is bounded, uniformly for $x\in B'.$

If there exists a function $w : {\mathbb R}^{n} \rightarrow (0,\infty)$ such that $hw\in L^{1}({\mathbb R}^{n})$ and 
\begin{align}\label{svajni}
\nu(x+y)\leq \nu(x)w(y),\quad x,\ y\in {\mathbb R}^{n},
\end{align} 
then the function 
\begin{align}\label{lpm}
(h\ast F)({\bf t};x):=\int_{{\mathbb R}^{n}}h(\sigma)F({\bf t}-\sigma;x)\, d\sigma,\quad {\bf t}\in {\mathbb R}^{n},\ x\in X
\end{align}
is $({\mathrm R}_{X},{\mathcal B},{\mathcal P},L)$-multi-almost periodic, resp. strongly $({\mathrm R}_{X},{\mathcal B},{\mathcal P},L)$-multi-almost periodic.
\end{prop}

\begin{thm}\label{nova}
Suppose that $(R({\bf t}))_{{\bf t}> {\bf 0}}\subseteq L(X,Y)$ is a strongly continuous operator family. 
Let $P:=C_{0,\nu}({\mathbb R}^{n} :Y)$ and $d(f,g):=\| f-g\|_{C_{0,\nu}({\mathbb R}^{n} :Y)}$ for all $f,\ g\in P,$ let there exist a positive real number $c>0$ such that $\nu({\bf t})\geq c$ for all ${\bf t}\in {\mathbb R}^{n},$ and let there exist a function $w : {\mathbb R}^{n} \rightarrow (0,\infty)$ such that \eqref{svajni} holds for all $x\in {\mathbb R}^{n},$ $y\in [0,\infty)^{n}$ and 
$\int_{(0,\infty)^{n}}(1+w({\bf t}))\|R({\bf t} )\|\, d{\bf t}<\infty .$ If $f : {\mathbb R}^{n} \rightarrow X$ is a bounded, continuous and $({\mathrm R},{\mathcal P},L)$-multi-almost periodic function, then the function $F: {\mathbb R}^{n} \rightarrow Y,$ given by
\begin{align}\label{wer}
F({\bf t}):=\int^{t_{1}}_{-\infty}\int^{t_{2}}_{-\infty}\cdot \cdot \cdot \int^{t_{n}}_{-\infty} R({\bf t}-{\bf s})f({\bf s})\, d{\bf s},\quad {\bf t}\in {\mathbb R}^{n},
\end{align}
is bounded, continuous and $({\mathrm R},{\mathcal P},L)$-multi-almost periodic.
\end{thm}

\begin{proof}
Let $({\bf b}_{k}=(b_{k}^{1},b_{k}^{2},\cdot \cdot\cdot ,b_{k}^{n})) \in {\mathrm R}$ be given. Then there exist a subsequence $({\bf b}_{k_{l}}=(b_{k_{l}}^{1},b_{k_{l}}^{2},\cdot \cdot\cdot , b_{k_{l}}^{n}))$ of $({\bf b}_{k})$ and a function
$f^{\ast} : {\mathbb R}^{n} \rightarrow X$ such that
\begin{align}\label{bgr}
\lim_{l\rightarrow +\infty}\Bigl[f\bigl(\cdot +(b_{k_{l}}^{1},\cdot \cdot\cdot, b_{k_{l}}^{n})\bigr)-f^{\ast}(\cdot)\Bigr]=0 
\end{align}
in $P.$ Since we have assumed that the function $f(\cdot)$ is bounded and there exists a positive real number $c>0$ such that $\nu({\bf t})\geq c$ for all ${\bf t}\in {\mathbb R}^{n},$ our choice of function space $P$ simply yields that
the function $f^{\ast} : {\mathbb R}^{n} \rightarrow X$ is bounded.
Let us prove that the function $f^{\ast}(\cdot)$ is continuous at the point ${\bf t}\in {\mathbb R}^{n}.$
If ${\bf t}'\in {\mathbb R}^{n}$ and $\epsilon>0,$ then we have
\begin{align*}
& \bigl\|  f^{\ast}({\bf t})-f^{\ast}\bigl({\bf t}'\bigr)  \bigr\|_{Y}
\\& \leq \bigl\| f^{\ast}({\bf t})-f\bigl( {\bf t}+b_{k_{l}} \bigr)\bigr\|_{Y} +\bigl\|f\bigl( {\bf t}+b_{k_{l}} \bigr)-f\bigl( {\bf t}'+b_{k_{l}} \bigr)\bigr\|_{Y}+\bigl\|f\bigl( {\bf t}'+b_{k_{l}} \bigr)-f^{\ast}\bigl({\bf t}'\bigr)\bigr\|_{Y}.
\end{align*}
The first addend and the third addend are less or equal than $\epsilon/3$ for some $l=l_{0}\in {\mathbb N}$, since \eqref{bgr} holds and $\nu({\bf t})\geq c$ for all ${\bf t}\in {\mathbb R}^{n},$ while the second addend is less or equal than $\epsilon/3$ whenever $\|f({\bf t}+b_{k_{l_{0}}})-f({\bf t}'+b_{k_{l_{0}}})\|_{Y}\leq \delta ,$ which is determined from the continuity of function $f(\cdot)$ at the point ${\bf t};$ this implies the claimed.
Moreover, we have 
\begin{align*}
F({\bf t})=\int_{[0,\infty)^{n}}R({\bf s}) f({\bf t}-{\bf s})\, d{\bf s}\mbox{ for all }{\bf t}\in {\mathbb R}^{n},
\end{align*}
the function
$F(\cdot)$ is bounded and continuous due to the dominated convergence theorem. 
The integral $ \int_{[0,\infty)^{n}}R({\bf s}) f^{\ast}({\bf t}-{\bf s})\, d{\bf s}$
is well defined for all ${\bf t}\in {\mathbb R}^{n},$ 
and we have:
$$
\lim_{l\rightarrow \infty} \int_{[0,\infty)^{n}}R({\bf s}) \bigl[f(\cdot+{\bf b}_{k_{l}}-{\bf s})- f^{\ast}(\cdot-{\bf s})\bigr]\, d{\bf s}=0
$$
in $P$, because 
\begin{align*}
\Biggl\|& \int_{[0,\infty)^{n}}R({\bf s})  f({\bf t}+{\bf b}_{k_{l}}-{\bf s})\, d{\bf s}- \int_{[0,\infty)^{n}}R({\bf s}) f^{\ast}({\bf t}-{\bf s})\, d{\bf s}\Biggr\|_{Y}\nu({\bf t})
\\& \leq \int_{[0,\infty)^{n}}\|R({\bf s})\| \cdot \bigl\| f({\bf t}+{\bf b}_{k_{l}}-{\bf s})- f^{\ast}({\bf t}-{\bf s}) \bigr\|\nu({\bf t})\, d{\bf s}
\\& \leq \int_{[0,\infty)^{n}}\|R({\bf s})\| \cdot \bigl\| f({\bf t}+{\bf b}_{k_{l}}-{\bf s})- f^{\ast}({\bf t}-{\bf s}) \bigr\|\nu({\bf t}-{\bf s})w({\bf s})\, d{\bf s}
\\& \leq \bigl\|f(\cdot+{\bf b}_{k_{l}}-{\bf s})- f^{\ast}(\cdot-{\bf s})\bigr\|_{P} \cdot \int_{[0,\infty)^{n}}\|R({\bf s})\|w({\bf s})\, d{\bf s},\quad {\bf t}\in {\mathbb R}^{n},\ l\in {\mathbb N}.
\end{align*}
We can similarly prove that
$$
\lim_{l\rightarrow \infty} \int_{[0,\infty)^{n}}R({\bf s}) \bigl[f^{\ast}(\cdot-{\bf b}_{k_{l}}-{\bf s})-f(\cdot-{\bf s})\bigr]\, d{\bf s}= 0
$$
in $P,$
which completes the proof. 
\end{proof}

Before proceeding any further, we feel it is our duty to say that the assumptions in Proposition \ref{convdiaggas}, resp. Theorem \ref{nova}, imply that the function $F(\cdot;\cdot)$ under our consideration is $({\mathrm R}_{X},{\mathcal B})$-multi-almost periodic, resp. $({\mathrm R},{\mathcal B})$-multi-almost periodic. Therefore, we deal here with certain subclasses of $({\mathrm R}_{X},{\mathcal B})$-multi-almost periodic functions, resp. $({\mathrm R},{\mathcal B})$-multi-almost periodic functions; see also Remark \ref{obeserve}(i) below, as well as Proposition \ref{convdiaggases} and Theorem \ref{novaes}, where we deal with certain subclasses of almost automorphic functions.

In \cite[Proposition 2.7, Proposition 2.8]{marko-manuel-ap}, we have examined the uniform convergence of  $({\mathrm R}_{\mathrm X},{\mathcal B})$-multi-almost periodic functions [$({\mathrm R},{\mathcal B})$-multi-almost periodic functions]. The situation is much more complicated with the notion introduced in this paper; concerning the above-mentioned issue, 
we will first note that,  in Definition \ref{eovakoapp-new}, we can additionally require that for each $l\in {\mathbb N}$ we also have $F(\cdot +(b_{k_{l}}^{1},\cdot \cdot\cdot, b_{k_{l}}^{n});x)\in P,$ $F^{\ast}(\cdot;x)\in P$, resp. $F(\cdot +(b_{k_{l}}^{1},\cdot \cdot\cdot, b_{k_{l}}^{n});x)\in P,$ $F^{\ast}(\cdot;x)\in P$, $F^{\ast}(\cdot -(b_{k_{l}}^{1},\cdot \cdot\cdot, b_{k_{l}}^{n});x)\in P$ and $F(\cdot;x)\in P,$ if $I={\mathbb R}^{n}$ and we consider strongly $({\mathrm R},{\mathcal B},{\mathcal P},L)$-multi-almost periodic functions; if this is the case, then we say that the function $F(\cdot;\cdot)$ is $({\mathrm R},{\mathcal B},{\mathcal P},L)$-multi-almost periodic of type $1$, resp. strongly $({\mathrm R},{\mathcal B},{\mathcal P},L)$-multi-almost periodic of type $1$. The class of $({\mathrm R}_{X},{\mathcal B},{\mathcal P},L)$-multi-almost periodic functions of type $1$, resp. strongly $({\mathrm R}_{X},{\mathcal B},{\mathcal P},L)$-multi-almost periodic functions of type $1,$ is introduced analogously. Now we are able to state and prove the following result:

\begin{prop}\label{mackat}
Suppose that $P$ has a linear vector structure, $P$ is complete and the metric $d$ is translation invariant in the sense that $d(f+g,h+g)=d(f,h)$ whenever $f,\ h,\ f+g,\ h+g\in P.$ 
Suppose further that, for each integer $j\in {\mathbb N}$ the function $F_{j} : I \times X \rightarrow Y$ is $({\mathrm R}_{\mathrm X}, {\mathcal B},{\mathcal P},L)$-multi-almost periodic of type $1$
as well as that, for every sequence  which belongs to 
${\mathrm R}_{\mathrm X},$ any its subsequence also belongs to ${\mathrm R}_{\mathrm X}.$ 
If $F : I \times X \rightarrow Y$ and for each $B\in {\mathcal B},$
$({\bf b};{\bf x})=((b_{k};x_{k}))=((b_{k}^{1},b_{k}^{2},\cdot \cdot\cdot ,b_{k}^{n});x_{k})\in {\mathrm R}_{X},$ $B'\in L(B;({\bf b};{\bf x}))$ and 
we have 
\begin{align}\label{konver}
\lim_{(i,l)\rightarrow +\infty} \sup_{x\in B'}\Bigl\| F_{i}\bigl(\cdot +b_{k_{l}};x+x_{k_{l}}\bigr)- F\bigl(\cdot +b_{k_{l}};x+x_{k_{l}}\bigr)\Bigr\|_{P}=0,
\end{align}
then the function $F(\cdot ;\cdot)$ is $({\mathrm R}_{\mathrm X}, {\mathcal B},{\mathcal P},L)$-multi-almost periodic of type $1.$
\end{prop}

\begin{proof}
Since $d$ is translation invariant, we have $\| f+g\|_{P}\leq \|f\|_{P}+\|g\|_{P}$ for all $f,\ g\in P$ such that $f+g\in P.$
Let the set $B\in {\mathcal B}$ and the sequence $(({\bf b};{\bf x})=((b_{k}^{1},b_{k}^{2},\cdot \cdot\cdot ,b_{k}^{n});x_{k}))\in {\mathrm R}_{\mathrm X}$ be given. Let $B'\in L(B;({\bf b};{\bf x}))$ be given, as well. Since 
for each sequence  which belongs to 
${\mathrm R}_{\mathrm X}$ we have that any its subsequence also belongs to ${\mathrm R}_{\mathrm X},$
the diagonal procedure can be used to get the existence of 
a subsequence $(({\bf b}_{k_{l}};x_{k_{l}})=((b_{k_{l}}^{1},b_{k_{l}}^{2},\cdot \cdot\cdot , b_{k_{l}}^{n});x_{k_{l}}))$ of $(({\bf b}_{k};x_{k}))$ such that for each integer $j\in {\mathbb N}$ there exists a function
$F^{\ast}_{j} : I \times X \rightarrow Y$ such that
\begin{align}\label{metalac12345}
\lim_{l\rightarrow +\infty}\sup_{x\in B'}\Bigl\| F_{j}\bigl(\cdot +(b_{k_{l}}^{1},\cdot \cdot\cdot, b_{k_{l}}^{n});x+x_{k_{l}}\bigr)-F^{\ast}_{j}(\cdot;x) \Bigr\|_{P}=0.
\end{align}
Let a real number $\epsilon>0$ be fixed.
Since 
\begin{align*}
\sup_{x\in B'}\Bigl\| F_{i}^{\ast}(\cdot ; x) &-F_{j}^{\ast}(\cdot ; x)\Bigr\|_{P}\leq \sup_{x\in B'} \Bigl\| F_{i}^{\ast}(\cdot ; x)-F_{i}\bigl(\cdot+(b_{k_{l}}^{1},\cdot \cdot\cdot, b_{k_{l}}^{n}) ; x+x_{k_{l}}\bigr)\Bigr\|_{P}
\\&+\sup_{x\in B'}\Bigl\| F_{i}\bigl(\cdot +(b_{k_{l}}^{1},\cdot \cdot\cdot, b_{k_{l}}^{n}); x+x_{k_{l}}\bigr)-F_{j}\bigl(\cdot +(b_{k_{l}}^{1},\cdot \cdot\cdot, b_{k_{l}}^{n}); x+x_{k_{l}}\bigr)\Bigr\|_{P}
\\&+\sup_{x\in B'}\Bigl\| F_{j}\bigl(\cdot +(b_{k_{l}}^{1},\cdot \cdot\cdot, b_{k_{l}}^{n}); x+x_{k_{l}}\bigr)-F_{j}^{\ast}(\cdot ; x)\Bigr\|_{P},
\end{align*}
and \eqref{metalac12345} holds, there exists an integer $l_{0}\in {\mathbb N}$ such that for all integers $l\geq l_{0}$ we have:
\begin{align*}
\sup_{x\in B'}\Bigl\| & F_{i}^{\ast}(\cdot ; x)-F_{i}\bigl(\cdot+(b_{k_{l}}^{1},\cdot \cdot\cdot, b_{k_{l}}^{n}) ; x+x_{k_{l}}\bigr)\Bigr\|_{P}
\\&+\sup_{x\in B'}\Bigl\| F_{j}\bigl(\cdot+(b_{k_{l}}^{1},\cdot \cdot\cdot, b_{k_{l}}^{n}); x+x_{k_{l}}\bigr)-F_{j}^{\ast}(\cdot ; x)\Bigr\|_{P}<2\epsilon/3.
\end{align*}
Using \eqref{konver}, we get the existence of an integer $N(\epsilon)\in {\mathbb N}$ such that
for all integers $i,\ j\in {\mathbb N}$ with $\min(i,j)\geq N(\epsilon)$ we have
\begin{align*}
\sup_{x\in B'}\Bigl\| F_{i}\bigl(\cdot +(b_{k_{l_{0}}}^{1},\cdot \cdot\cdot, b_{k_{l_{0}}}^{n}); x+x_{k_{l_{0}}}\bigr)-F_{j}\bigl(\cdot +(b_{k_{l_{0}}}^{1},\cdot \cdot\cdot, b_{k_{l_{0}}}^{n}); x+x_{k_{l_{0}}}\bigr)\Bigr\|_{P}<\epsilon/3.
\end{align*}
This implies that $(F_{j}^{\ast}(\cdot ;x))$ is a Cauchy sequence in $P$ for each element $x\in B'$ and therefore convergent to a function $F^{\ast}(\cdot ;x)\in P,$ say; clearly, 
$$
\lim_{j\rightarrow +\infty}\sup_{x\in B'} \bigl\|F_{j}^{\ast}(\cdot ;x)-F^{\ast}(\cdot ;x)\bigr\|_{P}=0.
$$ 
Further on, observe that for each $j\in {\mathbb N}$ we have:
\begin{align*}
\sup_{x\in B'} &\Bigl\| F\bigl(\cdot +(b_{k_{l}}^{1},\cdot \cdot\cdot, b_{k_{l}}^{n});x+x_{k_{l}}\bigr)-F^{\ast}(\cdot;x) \Bigr\|_{P}
\\& \leq \sup_{x\in B'} \Bigl\| F\bigl(\cdot +(b_{k_{l}}^{1},\cdot \cdot\cdot, b_{k_{l}}^{n});x+x_{k_{l}}\bigr)-F_{j}\bigl(\cdot +(b_{k_{l}}^{1},\cdot \cdot\cdot, b_{k_{l}}^{n});x+x_{k_{l}}\bigr) \Bigr\|_{P}
\\& +\sup_{x\in B'}\Bigl\| F_{j}\bigl(\cdot +(b_{k_{l}}^{1},\cdot \cdot\cdot, b_{k_{l}}^{n});x+x_{k_{l}}\bigr)-F^{\ast}_{j}(\cdot;x) \Bigr\|_{P}
+\sup_{x\in B'} \Bigl\| F^{\ast}_{j}(\cdot;x) - F^{\ast}(\cdot;x) \Bigr\|_{P}.
\end{align*}
There exists a number $j_{0}(\epsilon)\in {\mathbb N}$ such that for all integers $j\geq j_{0}$ we have
that the third addend in the above estimate is less or greater than $\epsilon/3.$ After that, using condition \eqref{konver}, we can find an integer $N\in {\mathbb N}$ such that $N\geq j_{0}$ and, for every $l\in {\mathbb N}$ and $j\in {\mathbb N}$ with $\min(j,l)\geq N,$ we have  
$$
\sup_{x\in B'} \Bigl\| F_{j}\bigl(\cdot+(b_{k_{l}}^{1},\cdot \cdot\cdot, b_{k_{l}}^{n});x+x_{k_{l}}\bigr)-F^{\ast}_{j}(\cdot;x) \Bigr\|_{P}<\epsilon/3. 
$$
Take $j=N$ and apply \eqref{metalac12345} to complete the
proof.
\end{proof}

\begin{rem}\label{rem1}
Consider the situation of Theorem \ref{nova}. If we assume that $f : {\mathbb R}^{n} \rightarrow X$ is a bounded, continuous and $({\mathrm R},{\mathcal P},L)$-multi-almost periodic function of type $1$, then the function $F: {\mathbb R}^{n} \rightarrow Y,$ given by \eqref{wer},
will be bounded, continuous and $({\mathrm R},{\mathcal P},L)$-multi-almost periodic of type $1$.
\end{rem}

We can simply prove the following:
\begin{itemize}
\item[(i)] Suppose that $c\in {\mathbb C},$ $cf\in P$ for all $f\in P,$ and there exists a finite real number $\phi(c)>0$ such that $\|cf\|_{P}\leq \phi(c)\|f\|_{P}$ for all $f\in P.$
If the function $F : I\times X \rightarrow Y$ is 
(strongly) $({\mathrm R},{\mathcal B},{\mathcal P},L)$-multi-almost periodic [(strongly) $({\mathrm R}_{X},{\mathcal B},{\mathcal P},L)$-multi-almost periodic], then the function $cF(\cdot;\cdot)$ is likewise 
(strongly) $({\mathrm R},{\mathcal B},{\mathcal P},L)$-multi-almost periodic [(strongly) $({\mathrm R}_{X},{\mathcal B},{\mathcal P},L)$-multi-almost periodic].
\item[(ii)] Suppose that $P$ has a linear vector structure, the metric $d$ is translation invariant and, for every complex number $c\in {\mathbb C},$ we have $cf\in P,$ $f\in P$ and there exists a finite real number $\phi(c)>0$ such that $\|cf\|_{P}\leq \phi(c)\|f\|_{P}$ for all $f\in P.$
Suppose also that for each sequence of collection ${\mathrm R}$ [${\mathrm R}_{X}$] any its subsequence also belongs to ${\mathrm R}$ [${\mathrm R}_{X}$]. Then the space consisting of all (strongly) $({\mathrm R},{\mathcal B},{\mathcal P},L)$-multi-almost periodic [(strongly) $({\mathrm R}_{X},{\mathcal B},{\mathcal P},L)$-multi-almost periodic] functions is a vector space.
\item[(iii)] Suppose that $\tau \in {\mathbb R}^{n}$, $\tau +I\subseteq I,$ $x_{0}\in X$ and, for every $f\in P,$ we have $f(\cdot+\tau)\in P$ and the existence of a finite real number $c_{\tau}>0$ such that $\| f(\cdot +\tau)\|_{P}\leq c_{\tau}\|f\|_{P}$ for all $f\in P.$ Define 
${\mathcal B}_{x_{0}}:= \{-x_{0}+B : B\in {\mathcal B}\}$ and, for every $B\in {\mathcal B}$ and ${\bf b}\in {\mathrm R}$ [$({\bf b};{\bf x})\in {\mathrm R}_{X}$], $L'(-x_{0}+B;{\bf b}):=\{
-x_{0}+B' : B'\in L(B;{\bf b})\}$ [$L'(-x_{0}+B;({\bf b};{\bf x})):=\{
-x_{0}+B' : B'\in L(B;({\bf b},{\bf x}))\}$]. If the function $F : I\times X \rightarrow Y$ is (strongly) $({\mathrm R},{\mathcal B},{\mathcal P},L)$-multi-almost periodic [(strongly) $({\mathrm R}_{X},{\mathcal B},{\mathcal P},L)$-multi-almost periodic], then the function $F(\cdot+\tau;\cdot +x_{0})$ is (strongly) $({\mathrm R},{\mathcal B}_{x_{0}},{\mathcal P},L')$-multi-almost periodic [(strongly) $({\mathrm R}_{X},{\mathcal B}_{x_{0}},{\mathcal P},L')$-multi-almost periodic].
\end{itemize}

The pointwise products of $({\mathrm R},{\mathcal B})$-multi-almost periodic functions have been analyzed in \cite[Proposition 2.20]{marko-manuel-ap}; we would like to note that the pointwise products of (strongly) $({\mathrm R},{\mathcal B},{\mathcal P},L)$-multi-almost periodic functions can be analyzed under certain extra assumptions. Details can be left to the interested readers.

Concerning composition principles, we will prove only one result which corresponds to \cite[Theorem 2.47]{marko-manuel-ap}.
Suppose that $F : I \times X \rightarrow Y$ and $G : I \times Y \rightarrow Z$ are given functions. We analyze here the multi-dimensional Nemytskii operator
$W : I  \times X \rightarrow Z,$ defined by
$$
W({\bf t}; x):=G\bigl({\bf t} ; F({\bf t}; x)\bigr),\quad {\bf t} \in I,\ x\in X.
$$
Let ${\mathcal P}_{F}=(P_{F},d_{f}),$ ${\mathcal P}_{G}=(P_{G},d_{G})$ and ${\mathcal P}_{W}=(P_{W},d_{W})$ be three metric spaces consisting of certain functions from $Y^{I},$ $Z^{I}$
and $Z^{I},$ respectively.

\begin{thm}\label{eovakoonakoap}
Suppose that $F : I \times X \rightarrow Y$ is $({\mathrm R},{\mathcal B},{\mathcal P}_{F},L)$-multi-almost periodic, resp. strongly 
$({\mathrm R},{\mathcal B},{\mathcal P}_{F},L)$-multi-almost periodic in the case that $I={\mathbb R}^{n}$,
$G : I \times Y \rightarrow Z$, the metric $d_{W}$ is translation invariant,
for any sequence belonging to ${\mathrm R}$ any its subsequence belongs to ${\mathrm R}$,
as well as the following condition holds:
\begin{itemize}
\item[(i)] For every ${\bf b}\in {\mathrm R},$ for every $B\in {\mathcal B}$, for every functions $f : I \times X \rightarrow Y$ and $g : I \times X \rightarrow Y$, for every sequences of functions $(f_{l} : I \times X \rightarrow Y)$, $(g_{l} : I \times X \rightarrow Y)$ and $(h_{l} : I \times X \rightarrow Y)$, and for every set $B'\in L(B;{\bf b}),$ we can find a subsequence ${\bf b}'$
of ${\bf b},$ a subsequence $(f_{m})$ of $(f_{l}),$
a
function $G^{\ast} : I \times Y \rightarrow Z$ and a finite real constant $c>0$
such that
\begin{align}\label{pariz}
\lim_{m\rightarrow +\infty}\sup_{x\in B'}\Bigl\|G\bigl(\cdot + {\bf b}_{m}' ; f(\cdot;x)\bigr)-G^{\ast}(\cdot;f(\cdot;x))\Bigr\|_{P_{G}}=0,
\end{align}
and
\begin{align}\label{pariz123}
\sup_{x\in B'}\Bigl\|G\bigl(\cdot + {\bf b}_{m} ; g_{m}(\cdot;x)\bigr)-G(\cdot + {\bf b}_{m};g(\cdot;x))\Bigr\|_{P_{G}}\leq c\sup_{x\in B'}\Bigl\| g_{m}(\cdot;x)-g(\cdot;x)\Bigr\|_{P_{F}}, 
\end{align}
whenever $m\in {\mathbb N}$ and $g_{m}(\cdot;x)-g(\cdot;x)\in P_{F}$ 
for all $x\in B',$ resp. \eqref{pariz} and \eqref{pariz123} hold, as well as
\begin{align*}
\lim_{m\rightarrow +\infty}\sup_{x\in B'}\Bigl\|G^{\ast}\bigl(\cdot - {\bf b}_{m}' ; f_{m}(\cdot;x)\bigr)-G(\cdot;f_{m}(\cdot;x))\Bigr\|_{P_{G}}=0
\end{align*}
and
\begin{align*}
\sup_{x\in B'}\Bigl\|G\bigl(\cdot  ; h_{m}(\cdot;x)\bigr)-G(\cdot ;h(\cdot;x))\Bigr\|_{P_{G}}\leq c\sup_{x\in B'}\Bigl\| h_{m}(\cdot;x)-h(\cdot;x)\Bigr\|_{P_{F}}, 
\end{align*}
whenever $m\in {\mathbb N}$ and
$h_{m}(\cdot;x)-h(\cdot;x)\in P_{F}$
for all $x\in B'.$
\end{itemize}
Then the function $W(\cdot; \cdot)$ is $({\mathrm R},{\mathcal B},{\mathcal P}_{W},L)$-multi-almost periodic, resp. strongly $({\mathrm R},{\mathcal B},{\mathcal P}_{W},L)$-multi-almost periodic.
\end{thm}

\begin{proof}
We will consider only $({\mathrm R},{\mathcal B},{\mathcal P}_{F},L)$-multi-almost periodic functions.
Let the set $B\in {\mathcal B}$ and the sequence $({\bf b}_{k}) \in {\mathrm R}$ be given, and let $B'\in L(B;{\bf b}).$
By definition, there exist a subsequence $({\bf b}_{k_{l}})$ of $({\bf b}_{k})$ and a function
$F^{\ast} : I \times X \rightarrow Y$ such that, for every $l\in {\mathbb N}$ and $x\in B,$ we have 
$F(\cdot +{\bf b}_{k_{l}};x)-F^{\ast}(\cdot;x)\in P$, and \eqref{love12345678app-new} 
holds. Since $({\bf b}_{k_{l}})\in  {\mathrm R}$, we can use assumption (i), with $f=F^{\ast} $ and the sequence of functions $(f_{l}(\cdot;\cdot)=F(\cdot+{\bf b}_{k_{l}};\cdot))$,
to find 
a subsequence $({\bf b}_{k_{l_{m}}})$ of $({\bf b}_{k_{l}})$ and a function
$G^{\ast} : I \times Y \rightarrow Z$ such that conditions \eqref{pariz} and \eqref{pariz123} hold.
It suffices to show that
\begin{align*}
\lim_{m\rightarrow +\infty}\Bigl\| G\bigl(\cdot +{\bf b}_{k_{l_{m}}};F\bigl(\cdot +{\bf b}_{k_{l_{m}}};x\bigr)\bigr)-G^{\ast}(\cdot;F^{\ast}(\cdot;x)) \Bigr\|_{P_{W}}=0.
\end{align*}
Denote ${\bf \tau_{m}}:={\bf b}_{k_{l_{m}}}$ for all $m\in {\mathbb N}$. We have ($x\in B',$ $m\in {\mathbb N}$):
\begin{align*}
&\Bigl\| G\bigl(\cdot +{\bf \tau_{m}};F\bigl(\cdot +{\bf \tau_{m}};x\bigr)\bigr)-G^{\ast}(\cdot;F^{\ast}(\cdot;x)) \Bigr\|_{P_{W}}
\\& \leq  \Bigl\| G\bigl(\cdot+{\bf \tau_{m}};F\bigl(\cdot +{\bf \tau_{m}};x\bigr)\bigr)-G(\cdot+{\bf \tau_{m}};F^{\ast}(\cdot;x)) \Bigr\|_{P_{W}}
\\& + \Bigl\| G(\cdot+{\bf \tau_{m}};F^{\ast}(\cdot;x)) -G^{\ast}(\cdot;F^{\ast}(\cdot;x)) \Bigr\|_{P_{W}}
\\& \leq c\Bigl\|F\bigl(\cdot+{\bf \tau_{m}};x\bigr)-F^{\ast}(\cdot;x) \Bigr\|_{P_{F}}+\Bigl\| G(\cdot+{\bf \tau_{m}};F^{\ast}(\cdot;x)) -G^{\ast}(\cdot;F^{\ast}(\cdot;x)) \Bigr\|_{P_{W}},
\end{align*}
which tends to zero as $m\rightarrow +\infty.$
This simply completes the proof.
\end{proof}

\begin{rem}\label{polu}
Suppose, additionally, that $F : I \times X \rightarrow Y$ is $({\mathrm R},{\mathcal B},{\mathcal P}_{F},L)$-multi-almost periodic of type $1$, resp. strongly 
$({\mathrm R},{\mathcal B},{\mathcal P}_{F},L)$-multi-almost periodic of type $1$ in the case that $I={\mathbb R}^{n}.$ If we assume, in condition (i), that we also have $G(\cdot +{\bf b}_{m}';f(\cdot))\in P_{W}$ for all $f\in P_{F},$ resp. $G(\cdot ;f(\cdot))\in P_{W}$ for all $f\in P_{F},$ then the function $W(\cdot; \cdot)$ will be $({\mathrm R},{\mathcal B},{\mathcal P}_{W},L)$-multi-almost periodic of type $1$, resp. strongly $({\mathrm R},{\mathcal B},{\mathcal P}_{W},L)$-multi-almost periodic of type $1$.
\end{rem}

\subsection{Generalization of multi-dimensional (Stepanov) almost automorphy}\label{autom}

The main aim of this subsection is to explain how we can use the approach obeyed so far to generalize the important classes of multi-dimensional (Stepanov) almost automorphic functions. Moreover, we will provide a completely new characterization of compactly almost automorphic functions by choosing the pivot space $P$ to be a weighted $L^{p}$-space. 

Concerning multi-dimensional almost automorphic functions, we will clarify the following result (for simplicity, we assume here that the corresponding collection $L(B; {\bf b})$ [$L(B; ({\bf b;{\bf x}}) )$] is equal to $\{ \{x\} : x\in B\}$ for each set $B\in {\mathcal B}$ and each sequence ${\bf b}\in {\mathrm R}$ [$({\bf b};{\bf x})\in {\mathrm R}_{X}$]; observe also that the notion from Definition \ref{eovako} and Definition \ref{eovakoap1aa} can be extended using this approach with the function $L$) :

\begin{thm}\label{automor}
Suppose that $F : {\mathbb R}^{n} \times X \rightarrow Y$ is a continuous function as well as that $p\in {\mathcal P}({\mathbb R}^{n}),$ 
$\nu : I \rightarrow (0,\infty)$ is a Lebesgue measurable function
and $L(B; {\bf b})=\{ \{x\} : x\in B\}$ [$L(B; ({\bf b;{\bf x}}) )=\{ \{x\} : x\in B\}$] for each set $B\in {\mathcal B}$ and each sequence ${\bf b}\in {\mathrm R}$ [$({\bf b};{\bf x})\in {\mathrm R}_{X}$]. Let $P:=L^{p({\bf t})}_{\nu}({\mathbb R}^{n} : Y)$ and $d(f,g):=\| f-g\|_{L^{p({\bf t})}_{\nu}({\mathbb R}^{n}: Y)}$ for all $f,\ g\in P.$
\begin{itemize}
\item[(i)]
Suppose that $p\in D_{+}({\mathbb R}^{n}),$ $\nu \in L^{p({\bf t})}({\mathbb R}^{n} :Y)$ and 
the function $F(\cdot;\cdot)$ is half-$({\mathrm R},{\mathcal B})$-multi-almost automorphic, resp. $({\mathrm R},{\mathcal B})$-multi-almost automorphic. If the function $F(\cdot;x)$ is bounded for every fixed element $x\in X,$ then the function $F(\cdot)$ is  $({\mathrm R},{\mathcal B},{\mathcal P},L)$-multi-almost periodic, resp. strongly $({\mathrm R},{\mathcal B},{\mathcal P},L)$-multi-almost periodic.
\item[(ii)] Suppose that $p\in D_{+}({\mathbb R}^{n}),$ $\nu \in L^{p({\bf t})}({\mathbb R}^{n} :Y)$ and 
the function $F(\cdot;\cdot)$ is half-$({\mathrm R}_{X},{\mathcal B})$-multi-almost automorphic, resp. $({\mathrm R}_{X},{\mathcal B})$-multi-almost automorphic.
If the function $F(\cdot;\cdot)$ is bounded, then the function $F(\cdot)$ is  $({\mathrm R}_{X},{\mathcal B},{\mathcal P},L)$-multi-almost periodic, resp. strongly $({\mathrm R}_{X},{\mathcal B},{\mathcal P},L)$-multi-almost periodic.
\item[(iii)] If the function $F(\cdot;x)$ is uniformly continuous for every fixed element $x\in X,$ and the function $F(\cdot)$ is $({\mathrm R},{\mathcal B},{\mathcal P},L)$-multi-almost periodic, resp. strongly $({\mathrm R},{\mathcal B},{\mathcal P},L)$-multi-almost periodic, then the function $F(\cdot;\cdot)$ is half-$({\mathrm R},{\mathcal B})$-multi-almost automorphic, resp. $({\mathrm R},{\mathcal B})$-multi-almost automorphic.
\item[(iv)] If the function $F(\cdot;\cdot)$ is uniformly continuous and  $({\mathrm R}_{X},{\mathcal B},{\mathcal P},L)$-multi-almost periodic, resp. strongly $({\mathrm R}_{X},{\mathcal B},{\mathcal P},L)$-multi-almost periodic, then the function $F(\cdot;\cdot)$ is half-$({\mathrm R}_{X},{\mathcal B})$-multi-almost automorphic, resp. $({\mathrm R}_{X},{\mathcal B})$-multi-almost automorphic.
\end{itemize}
\end{thm}

\begin{proof}
We will prove only (ii) and (iv). Let $B\in {\mathcal B},$ $x\in B$ and $(({\bf b;{\bf x}})_{k}=((b_{k}^{1},b_{k}^{2},\cdot \cdot\cdot ,b_{k}^{n});x_{k})_{k}) \in {\mathrm R}_{\mathrm X}$
be given. We know that
there exist a subsequence $(({\bf b;{\bf x}})_{k_{l}}=((b_{k_{l}}^{1},b_{k_{l}}^{2},\cdot \cdot\cdot , b_{k_{l}}^{n});x_{k_{l}})_{k_{l}})$ of $(({\bf b};{\bf x})_{k})$ and a function
$F^{\ast} : I \times X \rightarrow Y$ such that \eqref{gospode} holds, resp. \eqref{gospode} and \eqref{gospodine} hold true. This immediately implies that the function $F^{\ast}(\cdot;\cdot)$ is bounded; since we have assumed that $p\in D_{+}({\mathbb R}^{n})$ and $\nu \in L^{p({\bf t})}({\mathbb R}^{n} :Y),$ the dominated convergence theorem (see Lemma \ref{aux}(iv)) implies that
\begin{align}\label{rajke}
\lim_{l\rightarrow +\infty}\Bigl\| \bigl[ F({\bf t}+b_{k_{l}};x+x_{k_{l}}) -F^{\ast}({\bf t};x)\bigr] \cdot \nu({\bf t}) \Bigr\|_{L^{p({\bf t})}({\mathbb R}^{n} : Y)}=0,
\end{align}
resp. \eqref{rajke} and
\begin{align*}
\lim_{l\rightarrow +\infty}\Bigl\| \bigl[ F^{\ast}({\bf t}-b_{k_{l}};x-x_{k_{l}}) -F({\bf t};x)\bigr] \cdot \nu({\bf t}) \Bigr\|_{L^{p({\bf t})}({\mathbb R}^{n} : Y)}=0
\end{align*}
are true.
Therefore, the function $F(\cdot;\cdot)$ is  $({\mathrm R}_{X},{\mathcal B},{\mathcal P},L)$-multi-almost periodic, resp. strongly $({\mathrm R}_{X},{\mathcal B},{\mathcal P},L)$-multi-almost periodic. To prove (iv), fix
$B\in {\mathcal B}$ and $(({\bf b;{\bf x}})_{k}=((b_{k}^{1},b_{k}^{2},\cdot \cdot\cdot ,b_{k}^{n});x_{k})_{k}) \in {\mathrm R}_{\mathrm X}.$ We know that there exist a subsequence $(({\bf b;{\bf x}})_{k_{l}}=((b_{k_{l}}^{1},b_{k_{l}}^{2},\cdot \cdot\cdot , b_{k_{l}}^{n});x_{k_{l}})_{k_{l}})$ of $(({\bf b};{\bf x})_{k})$ and a function
$F^{\ast} : I \times X \rightarrow Y$ such that, for every $l\in {\mathbb N}$ and $x\in B,$ we have $F(\cdot +(b_{k_{l}}^{1},\cdot \cdot\cdot, b_{k_{l}}^{n});x+x_{k_{l}})-F^{\ast}(\cdot;x)\in P$ and \eqref{love12345678ap1} holds
pointwisely for all $x\in B$,
resp. $F(\cdot +(b_{k_{l}}^{1},\cdot \cdot\cdot, b_{k_{l}}^{n});x+x_{k_{l}})-F^{\ast}(\cdot;x)\in P$, $F^{\ast}(\cdot -(b_{k_{l}}^{1},\cdot \cdot\cdot, b_{k_{l}}^{n});x-x_{k_{l}})-F(\cdot;x)\in P,$ \eqref{love12345678ap1} holds, and \eqref{love12345678ap1}-\eqref{love12345678ap1s} hold pointwisely for all $x\in B$.
Therefore,
\begin{align}\label{rajkec}
\lim_{l\rightarrow +\infty}\Bigl\| \bigl[ F({\bf t}+b_{k_{l}};x+x_{k_{l}}) -F^{\ast}({\bf t};x)\bigr] \cdot \nu({\bf t}) \Bigr\|_{L^{p({\bf t})}(K_{T} : Y)}=0,
\end{align}
resp. \eqref{rajkec} and
\begin{align*}
\lim_{l\rightarrow +\infty}\Bigl\| \bigl[ F^{\ast}({\bf t}-b_{k_{l}};x-x_{k_{l}}) -F({\bf t};x)\bigr] \cdot \nu({\bf t}) \Bigr\|_{L^{p({\bf t})}(K_{T} : Y)}=0
\end{align*}
are true for all $T\in {\mathbb N}.$ Then Lemma \ref{aux}(ii) implies that the above equalities hold with the function $p({\bf t})$ replaced with the constant function $1$ therein, so that there exists 
a set $N\subseteq {\mathbb R}^{n}$ of Lebesgue measure zero such that \eqref{gospode}, resp. \eqref{gospode} and \eqref{gospodine},
hold pointwisely for all $x\in B$ and ${\bf t}\in {\mathbb R}^{n}\setminus N.$ Using the uniform continuity of function $F(\cdot;\cdot)$ and the well known $3-\epsilon$ argument, we can simply show that the function $F^{\ast}(\cdot;\cdot)$ is uniformly continuous and the limit equality \eqref{gospode}, resp. the limit equalities  \eqref{gospode} and \eqref{gospodine}, hold pointwisely for all $x\in B$ and ${\bf t}\in {\mathbb R}^{n},$ since $( F({\bf t}+b_{k_{l}};x+x_{k_{l}}) )$ and $(F^{\ast}({\bf t}-b_{k_{l}};x-x_{k_{l}})) $ are Cauchy sequences in $Y$.
\end{proof}

\begin{cor}\label{mace-mace}
Suppose that the function $F: {\mathbb R}^{n} \rightarrow Y$ is measurable, $p\in D_{+}({\mathbb R}^{n}),$ $\nu \in L^{p({\bf t})}({\mathbb R}^{n} :Y),$
${\mathrm R}$ is the collection of all sequences in ${\mathbb R}^{n}$, ${\mathcal P}$ and $L$ are defined as above.
Then the following holds:
\begin{itemize}
\item[(i)] If $F(\cdot)$ is almost automorphic, then $F(\cdot)$ is strongly $({\mathrm R},{\mathcal P},L)$-multi-almost periodic.
\item[(ii)] If $F(\cdot)$ is
uniformly continuous, then $F(\cdot)$ is compactly almost automorphic if and only if $F(\cdot)$ is strongly $({\mathrm R},{\mathcal P},L)$-multi-almost periodic.
\end{itemize}
\end{cor}

In the sequel, if $L$ is a function as in the previous two statements, we will also simply say that $f(\cdot)$ is (strongly) $({\mathrm R},{\mathcal P})$-multi-almost periodic [(strongly) $({\mathrm R}_{X},{\mathcal P})$-multi-almost periodic]; it is not clear whether the converse in (i) holds true. Further on,
let $p_{1}\in D_{+}({\mathbb R}^{n}),$ $\nu \in L^{p_{1}({\bf t})}({\mathbb R}^{n} :Y),$ ${\mathcal P}$ and $L$ be defined as above. It is worth noting that we have recently analyzed various classes of multi-dimensional Stepanov almost automorphic functions in \cite[Section 8.2]{nova-selected} as well as that the class of essentially bounded Stepanov $(\Omega,p({\bf u}))$-$({\mathrm R},{\mathcal B})$-multi-almost automorphic functions and the class of essentially bounded Stepanov $(\Omega,p({\bf u}))$-$({\mathrm R}_{X},{\mathcal B})$-multi-almost automorphic functions can be viewed as special subclasses of $({\mathrm R},{\mathcal B},{\mathcal P}_{1},L)$-multi-almost periodic functions, resp. $({\mathrm R}_{X},{\mathcal B},{\mathcal P}_{1},L)$-multi-almost periodic functions; here, $\Omega$ is any compact subset of ${\mathbb R}^{n}$ with positive Lebesgue measure and $p\in D_{+}(\Omega)$. Strictly speaking, if the Bochner transform $\hat{F} : {\mathbb R}^{n} \times X \rightarrow L^{p({\bf u})}({\mathbb R}^{n} : Y)$, defined by $[\hat{F}({\bf t};x)]({\bf u}):=F({\bf t}+{\bf u};x),$ ${\bf t}\in {\mathbb R}^{n},$ $x\in X,$ ${\bf u}\in \Omega,$
is $({\mathrm R},{\mathcal B})$-multi-almost automorphic, resp. $({\mathrm R}_{X},{\mathcal B})$-multi-almost automorphic, and the function
$F : {\mathbb R}^{n} \times X \rightarrow Y$ satisfies that the function
$F(\cdot;x)$ is essentially bounded for every fixed element $x\in X$, resp. the function $F(\cdot;\cdot)$
is essentially bounded,
then, for every $B\in {\mathcal B}$ and for every sequence $({\bf b})_{k} \in {\mathrm R},$ 
resp.
for every $B\in {\mathcal B}$ and for every sequence $(({\bf b;{\bf x}})_{k}) \in {\mathrm R}_{\mathrm X},$ there exist a subsequence $({\bf b})_{k_{l}} \in {\mathrm R}$ of $({\bf b})_{k} \in {\mathrm R},$ resp. $(({\bf b;{\bf x}})_{k_{l}})$ of $(({\bf b};{\bf x})_{k})$ and a function
$F^{\ast} : {\mathbb R}^{n} \times X \rightarrow L^{p({\bf u})}({\mathbb R}^{n} : Y)$ such that
\begin{align}\label{gospode0123}
\lim_{l\rightarrow +\infty}\Bigl\|F\bigl({\bf t} +{\bf u}+(b_{k_{l}}^{1},\cdot \cdot\cdot, b_{k_{l}}^{n});x\bigr)-\bigl[F^{\ast}({\bf t};x)\bigr]({\bf u})\Bigr\|_{L^{p({\bf u})}({\mathbb R}^{n} : Y)}=0
\end{align}
and
\begin{align}\label{gospodine0123}
\lim_{l\rightarrow +\infty}
\Bigl\|\bigl[F^{\ast}({\bf t}-{\bf u}+(b_{k_{l}}^{1},\cdot \cdot\cdot, b_{k_{l}}^{n});x)\bigr]({\bf u})-F({\bf t}+{\bf u};x)\Bigr\|_{L^{p({\bf u})}({\mathbb R}^{n} : Y)}=0,
\end{align}
resp.
\begin{align}\label{gospode123}
\lim_{l\rightarrow +\infty}\Bigl\|F\bigl({\bf t} +{\bf u}+(b_{k_{l}}^{1},\cdot \cdot\cdot, b_{k_{l}}^{n});x+x_{k_{l}}\bigr)-\bigl[F^{\ast}({\bf t};x)\bigr]({\bf u})\Bigr\|_{L^{p({\bf u})}({\mathbb R}^{n} : Y)}=0
\end{align}
and
\begin{align}\label{gospodine123}
\lim_{l\rightarrow +\infty}
\Bigl\|\bigl[F^{\ast}({\bf t}-{\bf u}+(b_{k_{l}}^{1},\cdot \cdot\cdot, b_{k_{l}}^{n});x-x_{k_{l}})\bigr]({\bf u})-F({\bf t}+{\bf u};x)\Bigr\|_{L^{p({\bf u})}({\mathbb R}^{n} : Y)}=0,
\end{align}
hold pointwisely for all $x\in B$ and ${\bf t}\in {\mathbb R}^{n}.$ The convergence in \eqref{gospode0123} and \eqref{gospodine0123}, resp. \eqref{gospode123} and \eqref{gospodine123}, implies the pointwise convergence for all ${\bf u}\in \Omega \setminus N_{{\bf t}},$ where $m(N_{{\bf t}})=0$. We can therefore define the limit function $F^{\ast} : {\mathbb R}^{n} \times X \rightarrow Y$ by
$F^{\ast}({\bf t};x):=[\hat{F}({\bf t}-{\bf u};x)]({\bf u})$, ${\bf t}\in {\mathbb R}^{n},$ $x\in X$, where ${\bf u} \in \Omega \setminus N_{{\bf t}}$ for all ${\bf t}\in {\mathbb R}^{n}.$ The function $F^{\ast}(\cdot;x)$ is essentially bounded for every fixed element $x\in X,$ resp. the function $F^{\ast}(\cdot;\cdot)$ is essentially bounded; using the dominated convergence theorem, we get that the function $F(\cdot;\cdot)$ is 
$({\mathrm R},{\mathcal B},{\mathcal P}_{1},L)$-multi-almost periodic, resp. $({\mathrm R}_{X},{\mathcal B},{\mathcal P}_{1},L)$-multi-almost periodic. Therefore, we have proved the following (see also Subsection \ref{periodic123} below for another approach obeyed with regards to the multi-dimensional Stepanov almost periodic functions):

\begin{thm}\label{step}
Suppose that the function $F : {\mathbb R}^{n} \times X \rightarrow Y$ satisfies that the function $F(\cdot;x)$ is essentially bounded for every fixed element $x\in X,$ resp. the function $F(\cdot;\cdot)$ is essentially bounded. If the function $F(\cdot;\cdot)$ is Stepanov $(\Omega,p({\bf u}))$-$({\mathrm R},{\mathcal B})$-multi-almost automorphic, resp. Stepanov $(\Omega,p({\bf u}))$-$({\mathrm R}_{X},{\mathcal B})$-multi-almost automorphic, then the function $F(\cdot;\cdot)$ is $({\mathrm R},{\mathcal B},{\mathcal P}_{1},L)$-multi-almost periodic, resp. $({\mathrm R}_{X},{\mathcal B},{\mathcal P}_{1},L)$-multi-almost periodic, where $p_{1}\in D_{+}({\mathbb R}^{n}),$ $\nu \in L^{p_{1}({\bf t})}({\mathbb R}^{n} :Y),$ ${\mathcal P}$ and $L$ being defined as above.
\end{thm}

We close this section by providing several observations in the case that $\nu : I \rightarrow (0,\infty)$ satisfies that $1/\nu(\cdot)$ is locally bounded, $P:=C_{0,\nu}(I : Y)$ and $d(f,g):=\|f-g\|_{C_{0,\nu}(I : Y)}$ for all $f,\ g\in P:$

\begin{rem}\label{obeserve}
\begin{itemize}
\item[(i)] Suppose that there exists a positive real number $c>0$ such that $\nu ({\bf t})\geq c$ for all ${\bf t}\in I,$ and $B\in L(B; {\bf b})$ for all $B\in {\mathcal B}$ and ${\bf b}\in {\mathrm R}$ [$B\in L(B;({\bf b};{\bf x}))$ for all $B\in {\mathcal B}$ and $({\bf b};{\bf x}) \in {\mathrm R}_{X}$]. Then any $({\mathrm R},{\mathcal B},{\mathcal P},L)$-multi-almost periodic function [$({\mathrm R}_{X},{\mathcal B},{\mathcal P},L)$-multi-almost periodic function] is  $({\mathrm R},{\mathcal B})$-multi-almost periodic [$({\mathrm R}_{X},{\mathcal B})$-multi-almost periodic].
\item[(ii)] Suppose that there exists a positive real number $c>0$ such that $\nu ({\bf t})\leq c$ for all ${\bf t}\in I,$ and $B\in L(B; {\bf b})$ for all $B\in {\mathcal B}$ and ${\bf b}\in {\mathrm R}$ [$B\in L(B;({\bf b};{\bf x}))$ for all $B\in {\mathcal B}$ and $({\bf b};{\bf x}) \in {\mathrm R}_{X}$]. Then any  $({\mathrm R},{\mathcal B})$-multi-almost periodic function [$({\mathrm R}_{X},{\mathcal B})$-multi-almost periodic function] is
$({\mathrm R},{\mathcal P},L)$-multi-almost periodic [$({\mathrm R}_{X},{\mathcal P},L)$-multi-almost periodic].
\item[(iii)] Suppose that $I={\mathbb R}^{n}.$ Then any $({\mathrm R},{\mathcal P},L)$-multi-almost periodic function [$({\mathrm R}_{X},{\mathcal P},L)$-multi-almost periodic] function is compactly $({\mathrm R},{\mathcal B})$-multi-almost automorphic [compactly $({\mathrm R}_{X},{\mathcal B})$-multi-almost automorphic]. This almost immediately follows from the corresponding definitions and the fact that the function $1/\nu(\cdot)$ is bounded on any compact of ${\mathbb R}^{n}.$
\end{itemize}
\end{rem}

\section{Bohr $({\mathcal B},I',\rho, {\mathcal P})$-Multi-almost periodic type functions}\label{maremare1}

Arguing as in Remark \ref{net}(ii) above, we may simply conclude that 
the notion of Bohr $({\mathcal B},I',\rho)$-almost periodicity
and the notion of $({\mathcal B},I',\rho)$-uniform recurrence
are very special cases of the following notion:

\begin{defn}\label{nafaks123456789012345}
Suppose that $\emptyset  \neq I' \subseteq {\mathbb R}^{n},$ $\emptyset  \neq I \subseteq {\mathbb R}^{n},$ $F : I \times X \rightarrow Y$ is a given function, $\rho$ is a binary relation on $Y,$ and $I +I' \subseteq I.$ Then we say that:
\begin{itemize}
\item[(i)]\index{function!Bohr $({\mathcal B},I',\rho,{\mathcal P})$-almost periodic}
$F(\cdot;\cdot)$ is Bohr $({\mathcal B},I',\rho,{\mathcal P})$-almost periodic if and only if for every $B\in {\mathcal B}$ and $\epsilon>0$
there exists $l>0$ such that for each ${\bf t}_{0} \in I'$ there exists ${\bf \tau} \in B({\bf t}_{0},l) \cap I'$ such that, for every ${\bf t}\in I$ and $x\in B,$ there exists an element $y_{{\bf t};x}\in \rho (F({\bf t};x))$ such that $F(\cdot+{\bf \tau};x)-y_{\cdot;x} \in P$ for all $x\in B$, and
\begin{align*}
\sup_{x\in B} \Bigl \| F(\cdot+{\bf \tau};x)-y_{\cdot;x}\Bigr\|_{P} \leq \epsilon .
\end{align*}
\item[(ii)] \index{function!$({\mathcal B},I',\rho,{\mathcal P})$-uniformly recurrent}
$F(\cdot;\cdot)$ is $({\mathcal B},I',\rho,{\mathcal P})$-uniformly recurrent if and only if for every $B\in {\mathcal B}$ 
there exists a sequence $({\bf \tau}_{k})$ in $I'$ such that $\lim_{k\rightarrow +\infty} |{\bf \tau}_{k}|=+\infty$ and that, for every ${\bf t}\in I$ and $x\in B,$ there exists an element $y_{{\bf t};x}\in \rho (F({\bf t};x))$ such that $F(\cdot+{\bf \tau}_{k};x)-y_{\cdot;x} \in P$ for all $k\in {\mathbb N},$ $x\in B$ and
\begin{align*}
\lim_{k\rightarrow +\infty}\sup_{x\in B} \Bigl \|F(\cdot+{\bf \tau}_{k};x)-y_{\cdot;x}\Bigr\|_{P}=0.
\end{align*}
\end{itemize}
\end{defn}

If $I={\mathbb R}^{n},$ then we can also consider the notions of strong Bohr $({\mathcal B},I',\rho,{\mathcal P})$-almost periodicity and strong $({\mathcal B},I',\rho,{\mathcal P})$-uniform recurrence; we will skip all details for simplicity.
Further on, we omit the term ``${\mathcal B}$'' if $X=\{0\},$ the term ``$I'$'' if $I'=I,$ and the term ``$\rho$'' if $\rho={\mathrm I}.$ If $\rho=c{\mathrm I},$ then we also say that the function $F(\cdot;\cdot)$ is 
Bohr $({\mathcal B},I',c,{\mathcal P})$-almost periodic [Bohr $({\mathcal B},I',c,{\mathcal P})$-uniformly recurrent]; furthermore, if $c=-1,$ then we say that the function $F(\cdot;\cdot)$ is 
Bohr $({\mathcal B},I',{\mathcal P})$-almost anti-periodic [Bohr $({\mathcal B},I',{\mathcal P})$-uniformly anti-recurrent].

We continue with the following example:

\begin{example}\label{budala}
Suppose that $I =[0,\infty)$ or $I= {\mathbb R}.$
A measurable function
$\nu : I \rightarrow (0,\infty)$ is said to be an
admissible weight function\index{admissible weight function} if and only if there exist finite constants $M\geq 1$ and
$\omega \in {\mathbb R}$ such that $\nu(t)\leq Me^{\omega
|t'|}\nu(t+t')$ for all $t,\ t'\in I.$
Let $\nu : [0,\infty) \rightarrow (0,\infty)$ be an admissible weight function; then it is well known that the function $1/\nu(\cdot)$ is locally bounded (see \cite{nova-chaos} for more details about linear topological dynamics and hypercyclic strongly continuous semigroups on weighted function spaces). 
Recently,
Z. Yin and Y. Wei
have considered the weak recurrence of translation operators on weighted Lebesgue spaces and weighted continuous function spaces (\cite{yin}). In particular, these authors have shown that
the existence of a function $f \in Y,$ where $Y= L^{p}_{\nu}([0,\infty) : {\mathbb C})$ or $Y=C_{0,\nu}([0,\infty) : {\mathbb C}),$
satisfying that there exists a strictly increasing sequence $(\alpha_{n})$ of positive reals tending to plus infinity such that
$$
\lim_{n\rightarrow +\infty}\bigl\| f(\cdot+\alpha_{n}) -f(\cdot) \bigr\|_{Y}=0
$$
is equivalent to saying that $\liminf_{t\rightarrow +\infty}\nu(t)=0$; see also the preprint \cite{brian} by W. Brian and J. P. Kelly. In our language, this result can be reworded as follows: Suppose that $P:=Y$ and $d(f,g):=\|f-g\|_{Y}$ for all $f,\ g\in P.$
Then there exists a ${\mathcal P}$-uniformly recurrent function $f \in Y$ if and only if $\liminf_{t\rightarrow +\infty}\nu(t)=0.$ It is without scope of this paper to analyze the corresponding result for the general space $Y= L^{p}_{\nu}(I : {\mathbb C})$ or $Y=C_{0,\nu}(I : {\mathbb C}).$
\end{example}

We can simply prove that any multi-dimensional $(\omega,\rho)$-periodic function is Bohr $(I',\rho,{\mathcal P})$-almost periodic with $I'=\{ k\omega : k\in {\mathbb N}\}$, provided that there exists a positive integer $l\in {\mathbb N}$ such that $\rho^{l}={\mathrm I}.$ But, there exist some pathological cases in which the spaces of 
Bohr $({\mathcal B},I',\rho,{\mathcal P})$-almost periodic functions consist solely of $(\omega,\rho)$-periodic functions. Without going into full details,
we will provide only one example regarding this issue:

\begin{example}
Suppose that the function $1/\nu(\cdot)$ is locally bounded, $P=C_{0,\nu}({\mathbb R}^{n} : Y)$, the metric $d$ is induced by the norm in $P,$  $0\neq \omega \in {\mathbb R}^{n},$ $F: {\mathbb R}^{n}\rightarrow Y$, and the supremum formula 
$$
\sup_{{\bf t}\in {\mathbb R}^{n}}\| F({\bf t}+\omega)- F({\bf t})\|_{Y}=\sup_{{\bf t}\in {\mathbb R}^{n}; |{\bf t}|\geq a} \| F({\bf t} +\omega)-F({\bf t})\|_{Y}
$$ 
holds for all $a>0.$ If $\nu : {\mathbb R}^{n} \rightarrow (0,\infty)$ satisfies 
$\lim_{|{\bf t}|\rightarrow +\infty}\nu({\bf t})=+\infty,$ then $F(\cdot)$ is $\omega$-periodic (i.e., $F({\bf t}+\omega)=F({\bf t}),$ ${\bf t}\in {\mathbb R}^{n}$) whenever
$F(\cdot +\omega)-F(\cdot)\in P.$ In actual fact, if $\epsilon>0$ is a fixed number and the last inclusion holds, then there exists a finite real number $M>0$ such that the assumption $|{\bf t}|\geq M$ implies $\|F({\bf t}+\omega)- F({\bf t})\|_{Y}\leq \epsilon.$ Due to the supremum formula, this implies
$\|F({\bf t}+\omega)- F({\bf t})\|_{Y}\leq \epsilon$ for all ${\bf t}\in {\mathbb R}^{n},$
so that the claimed assertion follows from the fact that the number $\epsilon>0$ was arbitrary.
\end{example}

In the previous example, we have analyzed the extreme case $\lim_{|{\bf t}|\rightarrow +\infty}\nu({\bf t})=+\infty.$ The opposite extreme case $\lim_{|{\bf t}|\rightarrow +\infty}\nu({\bf t})=0$ is also important on account of the following:

\begin{thm}\label{sade}
\begin{itemize}
\item[(i)] Suppose that $F: {\mathbb R}^{n}\rightarrow Y$ is continuous, the function $1/\nu(\cdot)$ is locally bounded, $P=C_{0,\nu}({\mathbb R}^{n} : Y)$, and the metric $d$ is induced by the norm in $P.$ If 
$F(\cdot)$ is Bohr ${\mathcal P}$-almost periodic, then $F(\cdot)$ is almost automorphic, i.e., Levitan $N$-almost periodic and bounded.
\item[(ii)] Suppose that the requirements of \emph{(i)} hold as well as that $F: {\mathbb R}^{n}\rightarrow Y$ is bounded, the function $\nu(\cdot)$ is bounded, and $\lim_{|{\bf t}|\rightarrow +\infty}\nu({\bf t})=0.$ 
Then $F(\cdot)$ is almost automorphic if and only if 
$F(\cdot)$ is Bohr ${\mathcal P}$-almost periodic.
\end{itemize}
\end{thm} 

\begin{proof}
If $F(\cdot)$ is Bohr ${\mathcal P}$-almost periodic, $N>0$ and $\epsilon>0$ are given, then we have the existence of a finite real number $M>0$ such that
$1/\nu({\bf t})\leq M$ for $|{\bf t}|\leq N.$ If the requirements in Definition \ref{nafaks123456789012345} hold with the numbers $\epsilon/M$ and $\tau,$ then it can be simply proved that $\tau$ is a Levitan $(\epsilon,N)$-almost period of function $F(\cdot).$
To show that $F(\cdot)$ is bounded, take $\epsilon=1$ and find a real number $l>0$ such that, for every ${\bf t}_{0}\in {\mathbb R}^{n},$ the cube ${\bf t}_{0}+[0,l]^{n}$
contains a point $\tau$ such that 
\begin{align}\label{sadkraj}
\bigl\| F({\bf t}+\tau)-F({\bf t})\bigr\|_{Y}\nu({\bf t})\leq 1, \quad {\bf t}\in {\mathbb R}^{n}.
\end{align}
Let $\| F({\bf t})\|_{Y}\leq M$ and $1/\nu({\bf t})\leq L$ for ${\bf t}\in K_{\sqrt{n}l}.$ This implies
$$
\bigl\|  F({\bf t}+\tau)\bigr\|_{Y}\leq \| F({\bf t})\|_{Y}+\frac{1}{\nu({\bf t})}\leq M+L,\quad {\bf t}\in K_{\sqrt{n}l}.
$$ 
It can be simply shown that any point ${\bf t}\in {\mathbb R}^{n}$ can be written as ${\bf t}_{0}+\tau,$ where ${\bf t}_{0}\in K_{\sqrt{n}l}$ and $\tau$ satisfies \eqref{sadkraj}. Therefore, $\| F({\bf t})\|_{Y}\leq M+L$ for all ${\bf t}\in {\mathbb R}^{n}$ and (i) is proved. In order to see that (ii) holds,
suppose that $F(\cdot)$ is almost automorphic and $\epsilon>0.$ Then $F(\cdot)$ is Levitan $N$-almost periodic and there exists a finite real number $M>0$ such that the assumptions $|{\bf t}|\geq M$ and $\tau\in {\mathbb R}^{n}$ imply
$$
\bigl\| F({\bf t}+\tau)- F({\bf t})\bigr\|_{Y}\nu({\bf t})\leq 2\|F\|_{\infty}\nu({\bf t})\leq \epsilon.
$$ 
If $|{\bf t}|\leq M$ and $\tau\in {\mathbb R}^{n}$ is a Levitan $(\epsilon/\|\nu\|_{\infty},M)$-period of function $F(\cdot),$ then we have $
\| F({\bf t}+\tau)- F({\bf t})\|_{Y}\nu({\bf t})\leq \epsilon\nu({\bf t})/\|\nu\|_{\infty}\leq \epsilon;$ hence, $F(\cdot)$ is Bohr ${\mathcal P}$-almost periodic. The converse statement follows from (i). 
\end{proof}

Keeping in mind Theorem \ref{sade}, it is meaningful to define, for every almost automorphic function $F : {\mathbb R}^{n}\rightarrow Y$, the following space
\begin{align*}
{\mathrm F}_{F}:=\bigl\{ \nu : {\mathbb R}^{n} \rightarrow Y \, ; \, \mbox{ the function }1/\nu(\cdot)\mbox{ is locally bounded and the function }\\ F(\cdot)\mbox{ is Bohr }{\mathcal P}-\mbox{almost periodic} \bigr\}.
\end{align*}
Due to Theorem \ref{sade}(ii), we have that all bounded functions vanishing at infinity belong to the space ${\mathrm F}_{F}.$ This inclusion can be strict since, for every almost periodic function $F : {\mathbb R}^{n}\rightarrow Y,$ the space 
${\mathrm F}_{F}$ contains all functions bounded from above; furthermore, if the function $F(\cdot)$ is $({\bf \omega}_{j})_{j\in {\mathbb N}_{n}}$-periodic with some numbers
${\bf \omega}_{j}\in {\mathbb R} \setminus \{0\}$ ($1\leq j\leq n$), then the space ${\mathrm F}_{F}$ consists of all functions $\nu : {\mathbb R}^{n} \rightarrow (0,\infty) $ such that the function $1/\nu(\cdot)$ is locally bounded. It is natural to ask whether there exists an almost periodic function 
$F : {\mathbb R}^{n}\rightarrow Y$ which is not $({\bf \omega}_{j})_{j\in {\mathbb N}_{n}}$-periodic for some numbers
${\bf \omega}_{j}\in {\mathbb R} \setminus \{0\}$ ($1\leq j\leq n$) and which additionally satisfies that the space ${\mathrm F}_{F}$ contains an unbounded function?

We proceed with the following illustrative examples:

\begin{example}\label{nawr}
Let us recall that the function $t\mapsto F(t)\equiv 1/(2+\cos t+\cos (\sqrt{2}t)),$ $t\in {\mathbb R}$ is Levitan $N$-almost periodic and unbounded. Due to Theorem \ref{sade}(i), there is no function $\nu : {\mathbb R} \rightarrow (0,\infty)$ such that the function $1/\nu(\cdot)$ is locally bounded and $F(\cdot)$ is 
Bohr ${\mathcal P}$-almost periodic with the metric space ${\mathcal P}$ be defined as  above.
\end{example}

\begin{example}\label{novpr}
Let ${\mathrm F}$ denote the collection of all functions $\nu : {\mathbb R} \rightarrow (0,\infty)$ such that the function $1/\nu(\cdot)$ is locally bounded; if this the case, we denote $P_{\nu}=C_{0,\nu}({\mathbb R} : {\mathbb C})$ and define $d_{\nu}$ to be the metric induced by the norm in $P_{\nu}.$ 
We introduce the binary relation $\sim$ on ${\mathrm F}$ by: $\nu_{1} \sim \nu_{2}$ if and only if every Bohr ${\mathcal P}_{\nu_{1}}$-almost periodic function $F :  {\mathbb R}\rightarrow {\mathbb C}$ is Bohr ${\mathcal P}_{\nu_{2}}$-almost periodic and vice versa; clearly, $\sim$ is an equivalence relation.
If there exist two finite real constants $c_{1}>0$ and $c_{2}>0$ such that
$c_{1}\nu_{1}(t)\leq \nu_{2}(t)\leq c_{2}\nu_{1}(t)$
for all $t\in  {\mathbb R},$ then it is clear that $\nu_{1} \sim \nu_{2}$ (in particular, if there exist two finite real constants $c_{1}>0$ and $c_{2}>0$ such that
$c_{1}\leq \nu (t) \leq c_{2}$ for all $t\in  {\mathbb R},$ $P=C_{0,\nu}({\mathbb R} : {\mathbb C})$, and the metric $d$ is induced by the norm in $P,$ then the function $F(\cdot)$ is almost periodic if and only if the function $F(\cdot)$ is Bohr ${\mathcal P}$-almost periodic). We can use
Theorem \ref{sade} to show that the existence of finite real constants $c_{1}>0$ and $c_{2}>0$ such that $c_{1}\nu_{1}(t)\leq \nu_{2}(t)\leq c_{2}\nu_{1}(t),$
$t\in {\mathbb R}$ is only sufficient but not necessary for relation $\nu_{1} \sim \nu_{2}$ to be satisfied. In actual fact, put $\nu_{1}(t):=1/(t^{2}+1)$ and $\nu_{2}(t):=1/(t^{4}+1)$ for all $t\in {\mathbb R}.$ Then $\nu_{1} \sim \nu_{2}$ due to  Theorem \ref{sade}(ii) but we cannot find finite real constants $c_{1}>0$ and $c_{2}>0$ such that $c_{1}\nu_{1}(t)\leq \nu_{2}(t)\leq c_{2}\nu_{1}(t),$
$t\in  {\mathbb R}.$
\end{example}

The proof of following result is very similar to the proof of \cite[Proposition 2.2]{rho}; we will include all relevant details for the sake of completeness:

\begin{prop}\label{prcko}
Suppose that $\emptyset  \neq I' \subseteq {\mathbb R}^{n},$ $\emptyset  \neq I \subseteq {\mathbb R}^{n},$  $I +I' \subseteq I$, and the function $F : I \times X \rightarrow Y$ is Bohr $({\mathcal B},I',\rho,{\mathcal P})$-almost periodic ($({\mathcal B},I',\rho,{\mathcal P})$-uniformly recurrent), where $\rho$ is a binary relation on $Y$ satisfying $R(F)\subseteq D(\rho)$ and $\rho(y)$ is a singleton for any $y\in R(F).$ If for each ${\bf \tau}\in I'$ we have $\tau +I=I,$ 
as well as $P$ has a linear vector structure, the metric $d$ is translation invariant and the following condition holds:
\begin{itemize}\label{ijk}
\item[(P)] ${\mathcal P}_{1}=(P_{1},d_{1})$ is a metric space, $c\in (0,\infty)$ and for every $f\in P$ and $\tau \in I'$ we have $f(\cdot-\tau)\in P_{1}$ and $\| f(\cdot-\tau)\|_{P_{1}}\leq c\|f\|_{P}.$ 
\end{itemize}
Then $I+(I'-I')\subseteq I$ and the function $F(\cdot;\cdot)$ is Bohr $({\mathcal B},I'-I',{\rm I},{\mathcal P}_{1})$-almost periodic ($({\mathcal B},I'-I',{\rm I},{\mathcal P}_{1})$-uniformly recurrent).
\end{prop}

\begin{proof}
We will prove the statement for Bohr $({\mathcal B},I',\rho,{\mathcal P})$-almost periodic functions. As in \cite{rho}, we have $I+(I'-I')\subseteq I.$ Further on, let $\epsilon>0$ and $B\in {\mathcal B}$ be given. Then there exists $l>0$ such that for each ${\bf t}_{0}^{1},\  {\bf t}_{0}^{2} \in I'$ there exist two points ${\bf \tau}_{1} \in B({\bf t}_{0}^{1},l) \cap I'$ and
${\bf \tau}_{2} \in B({\bf t}_{0}^{2},l) \cap I'$
such that, for every $x\in B,$ we have
\begin{align*}
\bigl\|F\bigl(\cdot+{\bf \tau}_{1};x\bigr)-\rho(F(\cdot;x))\bigr\|_{P} \leq \epsilon/2 \ \ \mbox{ and }\ \ \bigl\|F\bigl(\cdot+{\bf \tau}_{2};x\bigr)-\rho(F(\cdot;x))\bigr\|_{P} \leq \epsilon/2.
\end{align*}
Since $P$ has a linear vector structure and the metric $d$ is translation invariant, the above implies
\begin{align*}
\bigl\|F\bigl(\cdot+{\bf \tau}_{1};x\bigr)-F\bigl(\cdot+{\bf \tau}_{2};x\bigr)\bigr\|_{P} \leq \epsilon,\ x\in B.
\end{align*}
Using condition (P) and translation for the vector $-\tau_{2},$ we get 
\begin{align*}
\bigl\|F\bigl(\cdot+\bigl[\tau_{2}-{\bf \tau}_{1}\bigr];x\bigr)-F\bigl(\cdot;x\bigr)\bigr\|_{P_{1}} \leq c\epsilon,\ x\in B.
\end{align*}
Since $\tau_{2}-\tau_{1}\in B({\bf t}_{0}^{2}-{\bf t}_{0}^{1},2l) \cap (I'-I'),$ this simply implies the claimed.
\end{proof}

\begin{cor}\label{rtanj}
Suppose that $\emptyset  \neq I' \subseteq {\mathbb R}^{n},$ and the function $F : {\mathbb R}^{n} \times X \rightarrow Y$ is Bohr $({\mathcal B},I',\rho,{\mathcal P})$-almost periodic ($({\mathcal B},I',\rho,{\mathcal P})$-uniformly recurrent), where $\rho$ is a binary relation on $Y$ satisfying $R(F)\subseteq D(\rho)$ and $\rho(y)$ is a singleton for any $y\in R(F).$ Suppose that $P$ has a linear vector structure, the metric $d$ is translation invariant and condition \emph{(P)} holds.
Then the function $F(\cdot;\cdot)$ is Bohr $({\mathcal B},I'-I',{\rm I},{\mathcal P}_{1})$-almost periodic ($({\mathcal B},I'-I',{\rm I},{\mathcal P}_{1})$-uniformly recurrent).
\end{cor}

The following example is a slight modification of \cite[Example 2.8]{rho}:

\begin{example}\label{pripaz}
Suppose that $\rho=T\in L(Y),$ $I=I'=[0,\infty)$ or $I=I'={\mathbb R},$ and $X=\{0\}.$ 
Clearly, for each $t\in I,$ $\tau \in I$ and $l\in {\mathbb{N}},$ we have
\begin{align*}
F\bigl(t +l\tau \bigr)-T^{l}F(t) 
=\sum_{j=0}^{l-1}T^{j}\Bigl[F\bigl(t+(l-j)\tau \bigr)-TF\bigl(t+(l-j-1)\tau \bigr)\Bigr].
\end{align*}
Let for each $f\in P$ and $\tau \in I$ we have $f(\cdot+\tau)\in P,$ $\| f(\cdot+\tau)\|_{P}=\|f\|_{P}$, $Tf\in P$ and $\| Tf\|_{P}\leq c_{T}\|f\|_{P}$ for some finite real constant $c_{T}>0$ independent of $f\in P.$ 
This implies that the function
$F(\cdot)$ is $(T^{l},{\mathcal P})$-almost periodic ($(T^{l},{\mathcal P})$-uniformly recurrent), provided that $F(\cdot)$ is $(T,{\mathcal P})$-almost periodic ($(T,{\mathcal P})$-uniformly recurrent). In particular, the function $F(\cdot)$ is ${\mathcal P}$-almost periodic (${\mathcal P}$-uniformly recurrent), provided that $F(\cdot)$ is $(T,{\mathcal P})$-almost periodic ($(T,{\mathcal P})$-uniformly recurrent) and there exists a positive integer $l\in {\mathbb N}$ such that $T^{l}={\rm I};$ a similar statement can be formulated for $(T,{\mathcal P})$-almost anti-periodic ($(T,{\mathcal P})$-uniformly anti-recurrent) functions.

Consider now the sum $f(\cdot)$ of functions $f_{1}(\cdot)$ and $f_{2}(\cdot)$ from Example \ref{stojko}(i). In our concrete situation, $P=\{ f\in C_{b}({\mathbb R} : {\mathbb R}) \ ; \ \sup_{t\in {\mathbb R}}V(t;f)<\infty\}$ and $d(f,g):=\|f-g\|_{\infty}+\sup_{t\in {\mathbb R}}V(t;f-g),$ $f,\ g\in P,$ where $V(t;f)$ denotes the total variation of function $f(\cdot)$ on the segment $[t-1,t+1].$ If $f(\cdot)$ is $(c,{\mathcal P})$-almost periodic for some $c\in {\mathbb C} \setminus \{0\},$ then the function $f(\cdot)$ must be $c$-almost periodic, and therefore, we must have $c=\pm 1$ since the function $f(\cdot)$ is real-valued (\cite{nova-selected}). 
Since the metric space ${\mathcal P}$ satisfies all requirements from the first part of this example with $T=c{\mathrm I},$  $(-1,{\mathcal P})$-almost periodicity of $f(\cdot)$ would imply its  ${\mathcal P}$-almost periodicity, which is not the case. Therefore, there is no $c\in {\mathbb C} \setminus \{0\}$ such that
$f(\cdot)$ is $(c,{\mathcal P})$-almost periodic. We can similarly prove that there is no $c\in {\mathbb C} \setminus \{0\}$ such that the sum $f(\cdot)$ of functions $f_{1}(\cdot)$ and $f_{2}(\cdot)$ from Example \ref{stojko}(ii) is Lipschitz $c$-almost periodic, with the meaning clear.
\end{example}

Concerning Bohr $({\mathcal B},I',\rho,{\mathcal P})$-almost periodic functions in the finite-dimensional spaces, we will only note that the statement of \cite[Proposition 2.20]{rho} admits a reformulation in our framework provided that $P$ has a linear vector structure and the metric $d$ is translation invariant.

The proof of following simple proposition is trivial and therefore omitted:

\begin{prop}\label{als}
Suppose that $\emptyset  \neq I' \subseteq {\mathbb R}^{n},$ $\emptyset  \neq I \subseteq {\mathbb R}^{n},$ $\rho$ is a binary relation on $Y,$ and $I +I' \subseteq I.$ Suppose, further, that
$P_{1} \subseteq Y^{I},$ 
${\mathcal P}_{1}=(P_{1},d_{1})$ is a metric space, $P\subseteq P_{1}$, and there exists a finite real constant $c>0$ such that $\| f\|_{P_{1}}\leq c \| f\|_{P}$ for all $f\in P.$ 
Then the following holds:
\begin{itemize}
\item[(i)]
If $F : I \times X \rightarrow Y$ is Bohr $({\mathcal B},I',\rho,{\mathcal P})$-almost periodic [$({\mathcal B},I',\rho,{\mathcal P})$-uniformly recurrent], then $F(\cdot;\cdot)$ is Bohr $({\mathcal B},I',\rho,{\mathcal P}_{1})$-almost periodic [$({\mathcal B},I',\rho,{\mathcal P}_{1})$-uniformly recurrent].
\item[(ii)] Let \eqref{lepolepo} hold. If $F : I \times X \rightarrow Y$ is $({\mathrm R},{\mathcal B},{\mathcal P})$-multi-almost periodic, resp. strongly $({\mathrm R},{\mathcal B},{\mathcal P})$-multi-almost periodic in the case that $I={\mathbb R}^{n},$ then 
$F(\cdot;\cdot)$ is $({\mathrm R},{\mathcal B},{\mathcal P}_{1})$-multi-almost periodic, resp. strongly $({\mathrm R},{\mathcal B},{\mathcal P}_{1})$-multi-almost periodic.
\item[(iii)] Let \eqref{lepolepo121} hold. If $F : I \times X \rightarrow Y$ is $({\mathrm R}_{X},{\mathcal B},{\mathcal P})$-multi-almost periodic, resp. strongly $({\mathrm R}_{X},{\mathcal B},{\mathcal P})$-multi-almost periodic in the case that $I={\mathbb R}^{n},$ then 
$F(\cdot;\cdot)$ is $({\mathrm R}_{X},{\mathcal B},{\mathcal P}_{1})$-multi-almost periodic, resp. strongly $({\mathrm R}_{X},{\mathcal B},{\mathcal P}_{1})$-multi-almost periodic.
\end{itemize}
\end{prop}

We also have the following:
\begin{itemize}
\item[(i)] Suppose that $c\in {\mathbb C},$ $cf\in P$ for all $f\in P,$ and there exists a finite real number $\phi(c)>0$ such that $\|cf\|_{P}\leq \phi(c)\|f\|_{P}$ for all $f\in P.$
If the function $F : I\times X \rightarrow Y$ is 
Bohr $({\mathcal B},I',\rho,{\mathcal P})$-almost periodic [$({\mathcal B},I',\rho,{\mathcal P})$-uniformly recurrent], then the function $cF(\cdot;\cdot)$ is 
Bohr $({\mathcal B},I',\rho_{c},{\mathcal P})$-almost periodic [$({\mathcal B},I',\rho_{c},{\mathcal P})$-uniformly recurrent], where 
$$
\rho_{c}:=\Bigl\{ (y_{1},y_{2}) \in Y\times Y : (\exists {\bf t}\in I) \ (\exists x\in X)\ y_{1}=cF({\bf t};x)\mbox{ and }y_{2}\in c\rho(F({\bf t};x)) \Bigr\}.
$$
\item[(ii)] Suppose that $\tau \in {\mathbb R}^{n}$, $\tau +I\subseteq I,$ $x_{0}\in X$ and, for every $f\in P,$ we have $f(\cdot+\tau)\in P$ and the existence of a finite real number $c_{\tau}>0$ such that $\| f(\cdot +\tau)\|_{P}\leq c_{\tau}\|f\|_{P}$ for all $f\in P.$ Define 
${\mathcal B}_{x_{0}}:= \{-x_{0}+B : B\in {\mathcal B}\}$ for every $B\in {\mathcal B}.$ If the function $F : I\times X \rightarrow Y$ is Bohr $({\mathcal B},I',\rho,{\mathcal P})$-almost periodic [$({\mathcal B},I',\rho,{\mathcal P})$-uniformly recurrent],
then the function $F(\cdot+\tau;\cdot +x_{0})$ is Bohr $({\mathcal B}_{x_{0}},I',\rho,{\mathcal P})$-almost periodic [$({\mathcal B}_{x_{0}},I',\rho,{\mathcal P})$-uniformly recurrent].
\end{itemize}

We can illustrate the notion introduced in this section and the former section by slight modifications of \cite[Example 6.13, Example 6.15]{marko-manuel-ap}. In order to avoid any plagiarism, we will only rearrange \cite[Example 6.13(i)]{marko-manuel-ap} for our new purposes:

\begin{example}\label{duca}
Suppose that $F_{j} : X \rightarrow Y$ is a continuous function satisfying that for each $B\in {\mathcal B}$ we have $\sup_{x\in B}\| F_{j}(x) \|_{Y}<\infty .$ Suppose, further, that the complex-valued mapping $t\mapsto (\int_{0}^{t}f_{1}(s)\, ds,...,\int_{0}^{t}f_{n}(s)\, ds),$ $t\geq 0$ is ${\mathcal P}$-almost periodic ($1\leq j \leq n$), where $1/\nu(\cdot)$ is locally bounded function, $P:=C_{0,\nu}([0,\infty) : {\mathbb C}^{n})$ and $d(f,g):=\|f-g\|_{P}$ for all $f,\ g\in P$. Set
\begin{align*}
F\bigl(t_{1},\cdot \cdot \cdot,t_{n+1}; x\bigr):=\sum_{j=1}^{n}\int_{t_{j}}^{t_{j+1}}f_{j}(s)\, ds \cdot F_{j}(x)\ \mbox{ for all }x\in X \mbox{ and } t_{j}\geq 0,\ 1\leq j\leq n.
\end{align*}
Arguing as in the above-mentioned example, we get that, for every $B\in {\mathcal B},$ $t_{1},\ \tau_{1}\in [0,\infty);\ ...\ ;\ t_{n+1},\ \tau_{n+1}\in [0,\infty)$ and $\epsilon>0,$ 
\begin{align*}
\Bigl\|F\bigl(& t_{1}+\tau_{1},\cdot \cdot \cdot,t_{n+1}+\tau_{n+1}; x\bigr)-F\bigl(t_{1},\cdot \cdot \cdot,t_{n+1}; x\bigr)\Bigr\|_{Y}
\\& \leq   M\sum_{j=1}^{n}\Biggl\{\Biggl| \int^{t_{j+1}+\tau_{j+1}}_{0}f_{j}(s)\, ds-\int^{t_{j+1}}_{0} f_{j}(s)\, ds\Biggr|
\\ & + \Biggl|\int^{t_{j}+\tau_{j}}_{0}f_{j}(s)\, ds-\int^{t_{j}}_{0}f_{j}(s)\, ds\Biggr|\Biggr\},
\end{align*}
where $M=\sup_{x\in B, 1\leq j\leq n}\bigl\|F_{j}(x)\bigr\|_{Y}.$ This implies that the function $F(\cdot;\cdot)$ is Bohr $({\mathcal B},{\mathcal P}_{1})$-almost periodic, where $P_{1}:=C_{0,\nu_{1}}([0,\infty)^{n} : Y)$ with 
$$
\nu_{1}\bigl(t_{1},...,t_{n+1}\bigr):=\Biggl[ \frac{1}{\nu(t_{1})}+...+\frac{1}{\nu(t_{n+1})}\Biggr]^{-1},\quad t_{j}\geq 0 \ \ (1\leq j\leq n+1),
$$
and $d_{1}(f,g):=\|f-g\|_{P_{1}}$ for all $f,\ g\in P_{1}.$
\end{example}

We continue by observing that we can formulate and prove an analogue of \cite[Proposition 2.27(i)]{marko-manuel-ap} and Proposition \ref{mackat} for the function spaces introduced in Definition \ref{nafaks123456789012345}; on the other hand, the statement of \cite[Theorem 2.37]{marko-manuel-ap} which concerns the extensions of multi-dimensional almost periodic functions does not admit a satisfactory reformulation in our new framework. Proposition \ref{convdiaggas} and Theorem \ref{nova} can be reformulated as follows (we do not need here the existence of a finite real number $c>0$ such that $\nu({\bf t})\geq c$ for all ${\bf t}\in {\mathbb R}^{n}$):

\begin{prop}\label{convdiaggases}
Let $P:=C_{0,\nu}({\mathbb R}^{n} :Y)$ and $d(f,g):=\| f-g\|_{C_{0,\nu}({\mathbb R}^{n} :Y)}$ for all $f,\ g\in P.$ 
Suppose that  $c\in {\mathbb C}\setminus \{0\}$ and $\emptyset \neq I'\subseteq {\mathbb R}^{n},$ $h\in L^{1}({\mathbb R}^{n}),$ and $F : {\mathbb R}^{n} \times X \rightarrow Y$ is a continuous function satisfying that for each $B\in {\mathcal B}$ there exists a finite real number $\epsilon_{B}>0$ such that
$\sup_{{\bf t}\in {\mathbb R}^{n},x\in B^{\cdot}}\|F({\bf t},x)\|_{Y}<+\infty,$
where $B^{\cdot} \equiv B^{\circ} \cup \bigcup_{x\in \partial B}B(x,\epsilon_{B}).$ 
Then the function $
(h\ast F)(\cdot;\cdot)$, given by \eqref{lpm}, 
is well defined and for each $B\in {\mathcal B}$ we have $\sup_{{\bf t}\in {\mathbb R}^{n},x\in B^{\cdot}}\|(h\ast F)({\bf t};x)\|_{Y}<+\infty;$ furthermore, if $F(\cdot;\cdot)$ is 
Bohr $(I',c,{\mathcal P})$-almost periodic, then $
(h\ast F)(\cdot;\cdot)$ is Bohr $(I',c,{\mathcal P})$-almost periodic.
\end{prop}

\begin{thm}\label{novaes}
Suppose that $(R({\bf t}))_{{\bf t}> {\bf 0}}\subseteq L(X,Y)$ is a strongly continuous operator family, $c\in {\mathbb C}\setminus \{0\}$ and $\emptyset \neq I'\subseteq {\mathbb R}^{n}.$
Let $P:=C_{0,\nu}({\mathbb R}^{n} :Y)$ and $d(f,g):=\| f-g\|_{C_{0,\nu}({\mathbb R}^{n} :Y)}$ for all $f,\ g\in P,$ and let there exist a function $w : {\mathbb R}^{n} \rightarrow (0,\infty)$ such that \eqref{svajni} holds for all $x\in {\mathbb R}^{n},$ $y\in [0,\infty)^{n}$ and 
$\int_{(0,\infty)^{n}}(1+w({\bf t}))\|R({\bf t} )\|\, d{\bf t}<\infty .$ If $f : {\mathbb R}^{n} \rightarrow X$ is a bounded, continuous and Bohr $(I',c,{\mathcal P})$-almost periodic function, then the function $F: {\mathbb R}^{n} \rightarrow Y,$ given by \eqref{wer}, is
bounded, continuous and Bohr $(I',c,{\mathcal P})$-almost periodic.
\end{thm}

The proof of following proposition is also simple and therefore omitted (we assume the general requirements on the regions $I$ and $I'$, the binary relation $\rho$ and the collections ${\mathrm R},$ ${\mathrm R}_{X}$ from the introduced definitions):

\begin{prop}\label{aerosmith}
Suppose that $ {\mathcal P}_{Z}=(P_{Z},d_{Z})$ is a metric space, $\phi : {\mathcal P} \rightarrow {\mathcal P}_{Z}$ is uniformly continuous, the function $\phi(F_{1}(\cdot))-\phi(F_{2}(\cdot))$ belongs to $P_{Z}$ whenever $F_{1},\ F_{2}\in Y^{I}$ and $F_{1}-F_{2}\in P,$ and $\rho_{Z}$ is any binary relation on $Z$ such that
$\phi(\rho(F({\bf t};x)))\subseteq \rho_{Z}(\phi(F({\bf t};x)))$ for all ${\bf t}\in I$ and $x\in X.$
Suppose, further, that
$F : I \times X \rightarrow Y$ is (strongly) $({\mathrm R},{\mathcal B},{\mathcal P},L)$-multi-almost periodic [(strongly) $({\mathrm R}_{X},{\mathcal B},{\mathcal P},L)$-multi-almost periodic/Bohr $({\mathcal B},I',\rho,{\mathcal P})$-almost periodic/$({\mathcal B},I',\rho,{\mathcal P})$-uniformly recurrent].
Then $\phi \circ F : I \times X \rightarrow Z$ is (strongly) $({\mathrm R},{\mathcal B},{\mathcal P}_{Z},L)$-multi-almost periodic [(strongly) $({\mathrm R}_{X},{\mathcal B},{\mathcal P}_{Z},L)$-multi-almost periodic/Bohr $({\mathcal B},I',\rho_{Z},{\mathcal P})$-almost periodic/$({\mathcal B},I',\rho_{Z},{\mathcal P})$-uniformly recurrent].
\end{prop}

\subsection{Generalization of multi-dimensional (Stepanov) $\rho$-almost periodicity}\label{periodic123}

In this subsection, we will employ the approach obeyed in the former part of this paper to generalize the notion of Bohr $({\mathcal B},I',\rho)$-almost periodicity ($({\mathcal B},I',\rho)$-uniform recurrence)
and the notion of Stepanov $(\Omega,p({\bf u}))$-$({\mathcal B},I',\rho)$-almost periodicity (Stepanov $(\Omega,p({\bf u}),\rho)$-$({\mathcal B},\Lambda')$-uniform recurrence). We have the following:
\begin{itemize}
\item[(A1)]
Suppose that $\emptyset  \neq I' \subseteq {\mathbb R}^{n},$ $\emptyset  \neq I \subseteq {\mathbb R}^{n},$ $\rho$ is a binary relation on $Y,$ $I +I' \subseteq I,$ and
$F : I \times X \rightarrow Y$ is Bohr $({\mathcal B},I',\rho)$-almost periodic ($({\mathcal B},I',\rho)$-uniformly recurrent). Suppose, further, that the set $I$ is Lebesgue measurable as well as that for each $x\in X$ any selection of the multi-valued mapping ${\bf t}\mapsto \rho(F({\bf t};x)),$ ${\bf t}\in I$ is Lebesgue measurable. Let a function
$\nu : I \rightarrow (0,\infty)$ be Lebesgue measurable, as well. 
Suppose now that $\nu \in L^{p({\bf t})}(I:{\mathbb C})$. Then Lemma \ref{aux}(iii) yields that the Banach space $L^{\infty}(I: Y)$ is continuously embedded into $L^{p({\bf t})}_{\rho}(I: Y);$
applying Proposition \ref{als}(i), we easily get that the function $F(\cdot;\cdot)$ is Bohr $({\mathcal B},I',\rho,{\mathcal P})$-almost periodic ($({\mathcal B},I',\rho,{\mathcal P})$-uniformly recurrent),
where $P:=L^{p({\bf t})}_{\nu}(I: Y)$ and $d(f,g):=\|f-g\|_{L^{p({\bf t})}_{\rho}(I: Y)}$ for all $f,\ g\in P.$ We can also use the space $C_{0,\nu}(I : Y)$ here; see Remark \ref{obeserve}(ii).
\item[(A2)]
Suppose that $\emptyset  \neq I' \subseteq {\mathbb R}^{n},$ $\emptyset  \neq I \subseteq {\mathbb R}^{n},$ $\Omega
$ is a compact subset of ${\mathbb R}^{n}$ with positive Lebesgue measure, $I+\Omega \subseteq I$ and the function $F : I\times X \rightarrow Y$ is 
Stepanov $(\Omega,p({\bf u}))$-$({\mathcal B},I',\rho)$-almost periodic in the following sense (see \cite{ejmaa-2022} for more details):
For every $B\in {\mathcal B}$ and $\epsilon>0$
there exists $l>0$ such that for each ${\bf t}_{0} \in I'$ there exists ${\bf \tau} \in B({\bf t}_{0},l) \cap I'$ such that, for every ${\bf t}\in I$ and $x\in B,$  the mapping ${\bf u} \mapsto \rho (F({\bf t}+{\bf u};x)),$
${\bf u}\in \Omega$ is well defined, single valued and measurable, belongs to the space $ L^{p({\bf u})}(\Omega : Y)$ and
\begin{align}\label{prc}
\bigl\|F({\bf t}+{\bf \tau}+{\bf u};x)-\rho (F({\bf t}+{\bf u};x))\bigr\|_{L^{p({\bf u})}(\Omega : Y)} \leq \epsilon,\quad {\bf t}\in I,\ x\in B.
\end{align}
Suppose, further, that there exist a countable family $D$ and a collection $\{k_{d} \in I : d\in D\}$ such that
$I=\bigcup_{d\in D}(k_{d}+\Omega)$ and $m((k_{d_{1}}+\Omega) \cap (k_{d_{2}}+\Omega))=0$ for all $d_{1},\ d_{2}\in D$ such that $d_{1}\neq d_{2}.$
Define $p : I \rightarrow [1,\infty)$ by $p({\bf t}):=p({\bf t}-k_{d})$ if there exists a unique $d\in D$ such that ${\bf t}\in k_{d}+\Omega;$ otherwise, we set $p({\bf t}):=0.$ The last assumption implies that this mapping is well defined and measurable; furthermore, a simple argumentation shows that, for every $x\in X,$ we have:
\begin{align}
\notag \bigl\| [F({\bf t}+\tau;x)&-\rho (F({\bf t};x)) ] \cdot \rho({\bf t})\bigr\|_{L^{p({\bf t})}((k_{d}+\Omega) : Y)}
\\\label{argum} &=\bigl\| [F({\bf t}+k_{d}+\tau;x)-\rho (F({\bf t}+k_{d};x))] \cdot \rho({\bf t}+k_{d}) \bigr\|_{L^{p({\bf t})}(\Omega : Y)}.
\end{align}
We will assume that
\begin{align}\label{prc1}
S:=\sum_{d\in D}\mbox{ess sup} _{{\bf t}\in k_{d}+\Omega}\nu({\bf t})<+\infty.
\end{align}
Let $B\in {\mathcal B}$ and $\epsilon>0$ be fixed, and let $l>0$ be determined from the Stepanov $(\Omega,p({\bf u}))$-$({\mathcal B},I',\rho)$-almost periodicity of function $F(\cdot;\cdot)$; further on, let \eqref{prc} be satisfied.
Then, due to \eqref{argum} and a simple argumentation involving Lemma \ref{aux}(iii), we have:
\begin{align*}
&\bigl\| F({\bf t}+\tau;x)-\rho (F({\bf t};x)) \bigr\|_{L^{p({\bf t})}_{\nu}(I : Y)}
\\&\leq \sum_{d\in D}\bigl\| F({\bf t}+\tau;x)-\rho (F({\bf t};x)) \bigr\|_{L^{p({\bf t})}_{\nu}((k_{d}+\Omega) : Y)}
\\& =\sum_{d\in D}\bigl\| [F({\bf t}+\tau;x)-\rho (F({\bf t};x)) ] \cdot \nu({\bf t})\bigr\|_{L^{p({\bf t})}((k_{d}+\Omega) : Y)}
\\& =\sum_{d\in D}\bigl\| [F({\bf t}+k_{d}+\tau;x)-\rho (F({\bf t}+k_{d};x))] \cdot \nu({\bf t}+k_{d}) \bigr\|_{L^{p({\bf t})}(\Omega : Y)}
\\& \leq \sum_{d\in D}\bigl\| F({\bf t}+k_{d}+\tau;x)-\rho (F({\bf t}+k_{d};x)) \bigr\|_{L^{p({\bf t})}(\Omega : Y)}\cdot \mbox{ess sup} _{{\bf t}\in k_{d}+\Omega}\nu({\bf t})
\\& \leq \sum_{d\in D}\epsilon\cdot \mbox{ess sup} _{{\bf t}\in k_{d}+\Omega}\nu({\bf t})\leq \epsilon S.
\end{align*}
Since \eqref{prc1} is assumed, the above implies that the function $F(\cdot;\cdot)$ is Bohr $({\mathcal B},I',\rho,{\mathcal P})$-almost periodic, where $P:=L^{p({\bf t})}_{\nu}(I: Y)$ and $d(f,g):=\|f-g\|_{L^{p({\bf t})}_{\rho}(I: Y)}$ for all $f,\ g\in P.$

The above conclusion can be also formulated for Stepanov $(\Omega,p({\bf u}),\rho)$-$({\mathcal B},I')$-uniformly recurrent functions  (\cite{ejmaa-2022}).
\end{itemize}

The above particularly shows that the spaces of Stepanov $p$-almost periodic functions $F : {\mathbb R}^{n} \rightarrow Y$, where
$1\leq p<\infty,$
can be embedded into the corresponding spaces of Bohr ${\mathcal P}$-almost periodic functions. For example, 
in the case of consideration of one-dimensional Stepanov $p$-almost periodic functions, we have $\Omega=[0,1]$, $I'={\mathbb R}$ 
and $\rho={\rm I}.$ Then we can take $\nu(t):=1/(|t|^{\zeta}+1),$ $t\in {\mathbb R}$ since, in this case, \eqref{prc1} holds true. On the other hand, it is not clear how one can embed the space of all (equi-)Weyl-$p$-almost periodic functions $F : {\mathbb R}^{n} \rightarrow Y$ in some space of  Bohr ${\mathcal P}$-almost periodic functions (see \cite{nova-mono}-\cite{nova-selected} for the notion).

\section{Applications to the abstract Volterra integro-differential equations}\label{some1234554321}

The main aim of this section is to incorporate our results in the analysis of existence and uniqueness of the metrical almost periodic type solutions for some classes of abstract Volterra integro-differential equations. 
\vspace{0.1cm}

1. Let $Y$ be one of the spaces $L^{p}({\mathbb R}^{n}),$ $C_{0}({\mathbb R}^{n})$ or $BUC({\mathbb R}^{n}),$ where $1\leq p<\infty.$ Then we know that the Gaussian semigroup\index{Gaussian semigroup}
$$
(G(t)F)(x):=\bigl( 4\pi t \bigr)^{-(n/2)}\int_{{\mathbb R}^{n}}F(x-y)e^{-\frac{|y|^{2}}{4t}}\, dy,\quad t>0,\ f\in Y,\ x\in {\mathbb R}^{n},
$$
can be extended to a bounded analytic $C_{0}$-semigroup of angle $\pi/2,$ generated by the Laplacian $\Delta_{Y}$ acting with its maximal distributional domain in $Y;$ see e.g., \cite[Example 3.7.6]{a43}. Let $c\in {\mathbb C}\setminus \{0\}$ and $\emptyset \neq I'\subseteq {\mathbb R}^{n},$ and let
$F(\cdot)$ be bounded and $({\mathrm R},{\mathcal P})$-multi-almost periodic (Bohr $(I',c,{\mathcal P})$-almost periodic), where ${\mathcal P}$ is given in the formulation of Proposition \ref{convdiaggas} (Proposition \ref{convdiaggases}). Applying this result, we get that for each $t_{0}>0$ the function ${\mathbb R}^{n}\ni x\mapsto u(x,t_{0})\equiv (G(t_{0})F)(x) \in {\mathbb C}$
is likewise bounded and $({\mathrm R},{\mathcal P})$-multi-almost periodic (Bohr $(I',c,{\mathcal P})$-almost periodic).
We can  similarly analyze the Poisson semigroup here; see e.g.,
\cite[Example 3.7.9]{a43}.\vspace{0.1cm}

2. It is clear that we can use a combination of Theorem \ref{nova} (Theorem \ref{novaes}) and Theorem \ref{eovakoonakoap} in the analysis of metrical almost periodic solutions in time-variable for a large class of abstract fractional semi-linear inclusions. For instance, of concern is the following abstract semi-linear Cauchy inclusion:
\begin{align}\label{left}
D_{t,+}^{\gamma}u(t)\in -{\mathcal A}u(t)+f(t,u(t)),\ t\in {\mathbb R},
\end{align}
where $D_{t,+}^{\gamma}$ denotes the Weyl-Liouville fractional derivative of order $\gamma \in (0,1),$
$f : {\mathbb R} \times X \rightarrow X$ has certain properties and ${\mathcal A}$ is a closed multivalued linear operator on $X$ satisfying the following condition: 
\begin{itemize} \index{removable singularity at zero}
\item[(P)]
There exist finite constants $c,\ M>0$ and $\beta \in (0,1]$ such that\index{condition!(PW)}
$$
\Psi:=\Psi_{c}:=\Bigl\{ \lambda \in {\mathbb C} : \Re \lambda \geq -c\bigl( |\Im \lambda| +1 \bigr) \Bigr\} \subseteq \rho({\mathcal A})
$$
and
$$
\| R(\lambda : {\mathcal A})\| \leq M\bigl( 1+|\lambda|\bigr)^{-\beta},\quad \lambda \in \Psi ;
$$
\end{itemize}
see \cite{nova-mono} for the notion and more details. 
Define
\begin{align*}
T_{\nu}(t)x:=\frac{1}{2\pi i}\int_{\Gamma}(-\lambda)^{\nu}e^{\lambda t}\bigl(  \lambda -{\mathcal A} \bigr)^{-1}x\, d\lambda,\quad x\in X,\ t>0 \ (\nu>0),
\end{align*}
where $\Gamma$ is the upwards oriented curve $\lambda=-c(|\eta|+1)+i\eta$ ($\eta \in {\mathbb R}$), 
\begin{align*}
T_{\gamma,\nu}(t)x:=t^{\gamma \nu}\int^{\infty}_{0}s^{\nu}\Phi_{\gamma}( s)T_{0}\bigl( st^{\gamma}\bigr)x\, ds,\quad t>0,\ x\in X,\ \nu >-\beta,
\end{align*}
$$
S_{\gamma}(t):=T_{\gamma,0}(t)\mbox{ and }P_{\gamma}(t):=\gamma T_{\gamma,1}(t)/t^{\gamma},\quad t>0.
$$
Define also
\begin{align*}
R_{\gamma}(t):= t^{\gamma -1}P_{\gamma}(t),\quad t>0,\ x\in X.
\end{align*}
By a mild solution of \eqref{left}, we mean any $ X$-continuous function $u(\cdot)$ such that $u(t)= (\Lambda_{\gamma} u)(t),$ $t\in {\mathbb R},$ where
$$
t\mapsto (\Lambda_{\gamma} u)(t):=\int_{-\infty}^{t}R_{\gamma}(t-s)f(s,u(s))\, ds,\ t\in {\mathbb R}.
$$
Suppose now that ${\mathrm R}$ denotes the collection of all sequences in $[0,\infty),$ 
 $P:=C_{0,\nu}({\mathbb R} :X)$ and $d(f,g):=\| f-g\|_{C_{0,\nu}({\mathbb R} :X)}$ for all $f,\ g\in P.$ Let there exist a positive real number $c>0$ such that $\nu(t)\geq c$ for all $t\in {\mathbb R},$ and let there exist a function $w : {\mathbb R} \rightarrow (0,\infty)$ such that \eqref{svajni} holds for all $x\in {\mathbb R},$ $y\geq 0$ and 
$\int_{(0,\infty)}(1+w(t))\|R_{\gamma}(t)\|\, dt<\infty .$ It can be simply verified that the space ${\mathcal X}$ consisting of all bounded, continuous, $({\mathrm R},{\mathcal P})$-multi-almost periodic functions $f : {\mathbb R}^{n} \rightarrow X$ of type $1$ is a complete metric space equipped with the metric $d(\cdot,\cdot)=\|\cdot-\cdot\|_{P}$ (see also Proposition \ref{mackat}). Suppose, further, that
$ f : {\mathbb R} \times X \rightarrow X$ is continuous, $({\mathrm R},{\mathcal B},{\mathcal P})$-multi-almost periodic of type $1$ with $X\in {\mathcal B}$, and satisfies that $\sup_{t\in {\mathbb R};x\in B}\| f(t,x)\|<+\infty$ for each bounded subset $B$ of $X.$ If we assume that there exists a finite real constant $L>0$ such that  $\|f(t,x)-f(t,y)\| \leq L\|x-y\|$ for all $t\in {\mathbb R}$
 and $x,\ y\in X,$ as well as that 
$L\int_{(0,\infty)}w(t)\|R_{\gamma}(t)\|\, dt<1,$ then the mapping $\Lambda_{\gamma} : {\mathcal X} \rightarrow {\mathcal X}$ is well defined due to Theorem \ref{nova} and Theorem \ref{eovakoonakoap}; moreover, this mapping is a contraction and has a unique fixed point theorem according to the Banach contraction principle. Therefore, there exists a unique 
bounded, continuous solution of 
inclusion \emph{\eqref{left}} which is $({\mathrm R},{\mathcal P})$-multi-almost periodic of type $1.$ 

In particular, the above conclusions can be incorporated in the study of the following semi-linear fractional Poisson heat equation with Weyl-Liouville fractional derivatives in the space $X=L^{p}(\Omega):$
\[
D^{\gamma}_{t,+}[m(x)v(t,x)]=(\Delta -b )v(t,x) +f(t,v(t,x)),\quad t\in {\mathbb R},\ x\in {\Omega},
\]
where $1\leq p <\infty$, $\Omega$ is an open domain in ${\mathbb R}^{n}$ with smooth boundary, $m\in L^{\infty}(\Omega),$ $m(x)\geq 0$ for a.e. $x\in \Omega,$ $\gamma \in (0,1),$ $\Delta$ is the Dirichlet Laplacian and $b>0;$ 
possible applications can be also given to the higher-order differential operators in H\"older spaces (\cite{nova-mono}).
\vspace{0.1cm}

3.
The choice of space ${\mathcal P}$ used in the first application of this section enables one to reformulate the conclusions from \cite[Example 1.1]{marko-manuel-ap} in our new context,
provided that $\int_{{\mathbb R}^{n}}|E(t_{0},y)|w(y)\, dy <\infty.$ This can be incorporated in the analysis of metrical almost periodic solutions of the inhomogeneous heat equation in ${\mathbb R}^{n}.$ 

Similarly, with the same choice of space ${\mathcal P},$ we can analyze the existence and uniqueness of Bohr $(I',c,{\mathcal P})$-almost periodic ($(I',c,{\mathcal P})$-uniformly recurrent) solutions of the wave equation; from the classical theory of partial differential equations, we know that the solution of wave equation in ${\mathbb R},$ ${\mathbb R}^{2}$ and ${\mathbb R}^{3}$ is given by the famous d'Alembert formula, the Poisson formula and the Kirchhoff formula, respectively (see \cite[pp. 540--542]{nova-selected} for more details).
We close this section with the observation that, in the multi-dimensional setting, we can also consider the Hammerstein semi-linear integral equation of convolution type on ${\mathbb R}^{n};$ see the fourth application in \cite[Section 3]{marko-manuel-ap}.

\section{Conclusions and final remarks}\label{micho}

In this paper, we have investigated
multi-dimensional 
almost periodic type functions in general metric. We have provided a unification concept for the spaces of multi-dimensional (Stepanov) almost periodic functions and multi-dimensional (Stepanov) almost automorphic functions.
Some applications of our results to
the abstract Volterra integro-differential equations and the partial differential equations are also given.

It is our strong belief that this research article is just a beginning of several serious investigations of almost periodicity in special metrics. It is also worth noting that this is probably the first research article concerning metrical almost periodicity in the multi-dimensional setting. 

Concerning some subjects not considered in this paper, we would like to mention that we have recently analyzed,
in \cite{marko-manuel-ap}-\cite{marko-manuel-aa} and \cite{rho}, various notions of ${\mathbb D}$-asymptotical periodicity and ${\mathbb D}$-asymptotical automorphy in the multi-dimensional setting.  We will not analyze these topics here.

We can also consider the following notion:

\begin{defn}\label{strong-app} 
Suppose that $\emptyset  \neq I \subseteq {\mathbb R}^{n},$ $F : I \times X \rightarrow Y$ is a continuous function, and $P$ contains all trigonometric polynomials.
Then we say that $F(\cdot;\cdot)$ is strongly $({\mathcal B},{\mathcal P})$-almost periodic if and only if $F(\cdot;x) \in P$ for all $x\in X$ and for each $B\in {\mathcal B}$ there exists a sequence $(P_{k}^{B}({\bf t};x))$ of trigonometric polynomials 
such that 
$$
\lim_{k\rightarrow +\infty}\sup_{x\in B}\bigl\|P_{k}^{B}(\cdot;x)-F(\cdot;x)\bigr\|_{P}=0.
$$ 
Here, by a trigonometric polynomial $P : I\times X \rightarrow Y$ we mean any linear combination of functions like
\begin{align*}
e^{i[\lambda_{1}t_{1}+\lambda_{2}t_{2}+\cdot \cdot \cdot +\lambda_{n}t_{n}]}c(x),
\end{align*}
where $\lambda_{i}$ are real numbers ($1\leq i \leq n$) and $c: X \rightarrow Y$ is a continuous mapping.
\end{defn}\index{function!strongly $({\mathcal B},{\mathcal P})$-almost periodic}

Then the notion of a strongly ${\mathcal B}$-almost periodic function is obtained by  plugging that $P=l_{\infty}(I:Y).$ We will not further analyze here the structural properties of strongly $({\mathcal B},{\mathcal P})$-almost periodic functions. 

The following notion is also meaningful:

\begin{defn}\label{drasko-presings}
Let ${\bf \omega}\in {\mathbb R}^{n} \setminus \{0\},$ $\rho$ be a binary relation on $Y$ 
and 
${\bf \omega}+I \subseteq I$. A
function $F:I\rightarrow Y$ is said to be semi-$({\bf \omega},\rho,{\mathcal P})$-periodic if and only if there exists a sequence $(F_{k})$ of $({\bf \omega},\rho)$-periodic functions such that
\begin{align}\label{raj}
\lim_{k\rightarrow \infty}\bigl\| F_{k}-F\bigr\|_{P}=0.
\end{align}
\end{defn}

\begin{defn}\label{drasko-presing1s}
Let ${\bf \omega}_{j}\in {\mathbb R} \setminus \{0\},$ $\rho_{j}\in {\mathbb C} \setminus \{0\}$ is a binary relation on $Y$
and 
${\bf \omega}_{j}e_{j}+I \subseteq I$ ($1\leq j\leq n$). A 
function $F:I\rightarrow Y$ is said to be semi-$({\bf \omega}_{j},\rho_{j},{\mathcal P})_{j\in {\mathbb N}_{n}}$-periodic if and only if there exists a sequence $(F_{k})$ of $({\bf \omega}_{j},\rho_{j})_{j\in {\mathbb N}_{n}}$-periodic functions such that \eqref{raj} holds.
\end{defn} \index{function!$({\bf \omega}_{j},\rho_{j},{\mathcal P})_{j\in {\mathbb N}_{n}}$-periodic}

More details and structural results on semi-$({\bf \omega},\rho,{\mathcal P})$-periodic functions and semi-$({\bf \omega}_{j},\rho_{j},{\mathcal P})_{j\in {\mathbb N}_{n}}$-periodic functions will appear somewhere else. Multi-dimensional analogues of almost periodic functions in the Hausdorff metric, multi-dimensional analogues of almost periodic functions in variation, and multi-dimensional analogues of $L_{\alpha}$-almost periodic functions will be considered somewhere else, as well.

\end{document}